\documentclass[11pt,reqno]{amsart}
\usepackage{graphicx}
\usepackage{hyperref}
\usepackage{verbatim}

\newtheorem{thm}{Theorem}[section]
\newtheorem{lem}[thm]{Lemma}
\newtheorem{prop}[thm]{Proposition}
\newtheorem{cor}[thm]{Corollary}

\theoremstyle{definition}

\theoremstyle{remark}

\numberwithin{equation}{section}

\def\N{\mathbb{N}}

\def\R{\mathbb{R}}

\def\ra{\rightarrow}
\def\bs{\backslash}

\def\al{\alpha}

\def\ep{\epsilon}
\def\ka{\kappa}

\def\la{\lambda}

\def\om{\omega}
\def\Om{\Omega}
\def\de{\delta}

\begin{document}

\title[Transition fronts in nonlocal equations]{Transition fronts in nonlocal equations with time heterogeneous ignition nonlinearity}

\author{Wenxian Shen}
\address{Department of Mathematics and Statistics, Auburn University, Auburn, AL 36849}
\email{wenxish@auburn.edu}

\author{Zhongwei Shen}
\address{Department of Mathematics and Statistics, Auburn University, Auburn, AL 36849}
\email{zzs0004@auburn.edu}





\begin{abstract}
The present paper is devoted to the study of transition fronts in nonlocal reaction-diffusion equations with time heterogeneous nonlinearity of ignition type. It is proven that such an equation admits space monotone transition fronts with finite speed and space regularity in the sense of uniform Lipschitz continuity. Our approach is first constructing a sequence of approximating front-like solutions and then proving that the approximating solutions converge to a transition front. We take advantage of the idea of modified interface location, which allows us to characterize the finite speed of approximating solutions in the absence of space regularity, and leads directly to uniform exponential decaying estimates.
\end{abstract}

\maketitle

\tableofcontents


\section{Introduction}

Since the pioneering works of Fisher (see \cite{Fish37}) and Kolmogorov, Petrowsky and Piscunov (see \cite{KPP37}), the reaction-diffusion equation of the form
\begin{equation}\label{eqn-classical}
u_{t}=u_{xx}+f(t,x,u)
\end{equation}
has attracted a lot of attention, and a vast amount of literature has been carried out to the understanding of front propagation phenomenon in such an equation and its generalized forms. We refer to \cite{ArWe75,ArWe78,FiMc77,FiMc80,Kam76,Uch78,Xi00} and references therein for works in the homogeneous media, i.e., $f(t,x,u)=f(u)$. Recently, there are a lot of progress concerning \eqref{eqn-classical} in the heterogeneous media. We refer to \cite{BeHa02,DiHaZh14,MeRoSi10,MNRR09,Na14,NoRy09,NRRZ12,We02,Zl12} and references therein for works in the space heterogeneous media, i.e., $f(t,x,u)=f(x,u)$, and to \cite{AlBaCh99,NaRo12,Sh99-1,Sh99-2,Sh06,Sh11,ShSh14,ShSh14-1} and references therein for works in the time heterogeneous media, i.e., $f(t,x,u)=f(t,x)$. There are also some works in the space-time heterogeneous media (see e.g. \cite{KoSh14,LiZh1,LiZh2,Na09,Na10, Sh04,Sh11-1,We02}), but it remains widely open. It is worthwhile to point out that the stability results concerning traveling waves and their generalizations, from a PDE viewpoint, are more or less influenced by the work of Fife and McLeod (see \cite{FiMc77}).

Equation \eqref{eqn-classical} is used to model various diffusive processes in biology, ecology, combustion theory and so on. However, when processes with jumps come into play, equation \eqref{eqn-classical} is no longer suitable. More precisely, the random dispersal operator $\partial_{xx}$ is no long suitable. As a substitute, the nonlocal dispersal operator is introduced (see e.g. \cite{Fi03} for some background) and we are now concerned with the following integral equation
\begin{equation}\label{eqn-nonlocal}
u_{t}=J\ast u-u+f(t,x,u),
\end{equation}
where $J$ is a convolution kernel and $[J\ast u](x)=\int_{\R}J(x-y)u(y)dy=\int_{\R}J(y)u(x-y)dy$.

As for \eqref{eqn-nonlocal} in the homogeneous media, traveling waves, i.e., solutions of the form $u(t,x)=\phi(x-ct)$ with $(c,\phi)$ satisfying
\begin{equation*}
J\ast\phi-\phi+c\phi_{x}+f(\phi)=0,\quad \phi(-\infty)=1,\quad\phi(\infty)=0,
\end{equation*}
have been obtained for various nonlinearities, including bistable nonlinearity, monostable nonlinearity, and
 ignition type nonlinearity (see \cite{BaFiReWa97,CaCh04,Ch97,Cov-thesis,CoDu05,CoDu07,Sc80} and references therein). The study of \eqref{eqn-nonlocal} in the heterogeneous media is rather recent and results concerning front propagation are very limited comparing to that of the classical random dispersal case. In \cite{CDM13,DiLi15,ShZh10,ShZh12-1,ShZh12-2}, the authors investigated \eqref{eqn-nonlocal} in the space periodic monostable media, i.e., $f(t,x,u)=f(x,u)$ is of monostable type and periodic in $x$, and proved the existence of spreading speeds and periodic traveling waves.
 In \cite{RaShZh}, the authors studied the existence of spreading speeds and traveling waves of \eqref{eqn-nonlocal} in the space-time periodic monostable media. Very recently, both Berestycki, Coville and Vo (see \cite{BCV14}), and Lim and Zlato\v{s} (see \cite{LiZl14}) investigated \eqref{eqn-nonlocal} in the space heterogeneous monostable media. While Berestycki, Coville and Vo stuided principal eigenvalue, positive solution and long-time behavior of solutions, Lim and Zlato\v{s} proved the existence of transition fronts in the sense of Berestycki-Hamel (see \cite{BeHa07,BeHa12}), that is, an entire solution $u(t,x)$ of \eqref{eqn-classical} or \eqref{eqn-nonlocal} is called a \textit{transition front} if $u(t,-\infty)=1$ and $u(t,\infty)=0$ for any $t\in\R$, and for any $\ep\in(0,1)$ there holds
\begin{equation*}
\sup_{t\in\R}\text{diam}\{x\in\R|\ep\leq u(t,x)\leq 1-\ep\}<\infty.
\end{equation*}

As just mentioned, the existing results for \eqref{eqn-nonlocal} in the heterogeneous media is far from being complete, and therefore, it is expected to explore more results in time or space heterogeneous media of monostable, bistable or ignition type. In the present paper, we study \eqref{eqn-nonlocal} in the time heterogeneous media of ignition type, i.e.,
\begin{equation}\label{main-eqn}
u_{t}=J\ast u-u+f(t,u),\quad (t,x)\in\R\times\R,
\end{equation}
where the convolution kernel $J$ satisfies
\begin{itemize}
\item[\rm(H1)] $J\not\equiv0$, $J\in C^{1}(\R)$, $J(x)=J(-x)\geq0$ for all $x\in\R$, $\int_{\R}J(x)dx=1$, $\int_{\R}|J'(x)|dx<\infty$ and
\begin{equation}\label{decay-convolution-kernel}
\int_{\R}J(x)e^{\la x}dx<\infty,\quad\forall\la>0;
\end{equation}
\end{itemize}
and the time heterogeneous nonlinearity $f(t,u)$ satisfies
\begin{itemize}
\item[\rm(H2)] $f:\R\times[0,\infty)\ra\R$ is continuously differentiable and satisfies the following conditions:
\begin{itemize}
\item there are $\theta\in(0,1)$ (the ignition temperature), $f_{\min}\in C^{1,\al}([0,1])$ and a Lipschitz continuous function $f_{\max}:[0,1]\ra\R$ satisfying
\begin{equation*}
\begin{split}
f_{\min}(u)=0=f_{\max}(u),\,\,&u\in[0,\theta]\cup\{1\},\\
0<f_{\min}(u)\leq f_{\max}(u),\,\,&u\in(\theta,1),\\
f_{\min}'(1)<0
\end{split}
\end{equation*}
such that
\begin{equation*}
f_{\min}(u)\leq f(t,u)\leq f_{\max}(u),\quad (t,u)\in\R\times[0,1].
\end{equation*}

\item $f(t,u)<0$ for $(t,u)\in\R\times(1,\infty)$;

\item first-order partial derivatives are uniformly bounded, i.e.,
\begin{equation*}
\sup_{(t,u)\in\R\times[0,1]}|f_{t}(t,u)|<\infty\quad\text{and}\sup_{(t,u)\in\R\times[0,\infty)}|f_{u}(t,u)|<\infty
\end{equation*}

\item there exists $\tilde{\theta}\in(\theta,1)$ such that
\begin{equation}\label{decreasing-near-1}
\frac{f(t,u_{1})-f(t,u_{2})}{u_{1}-u_{2}}\leq0,\quad\forall u_{1},u_{2}\in[\tilde{\theta},1],\,\,u_{1}\neq u_{2},\,\,t\in\R.
\end{equation}
\end{itemize}
\end{itemize}

We remark that \eqref{decay-convolution-kernel} says that $J$ decays at $\pm\infty$ faster than any exponential function. Typical examples for $J$ are compactly supported functions and the probability density function for standard normal distribution. Under the continuous differentiability assumption of $f$, \eqref{decreasing-near-1} is equivalent to $f_{u}(t,u)\leq0$ for all $(t,u)\in\R\times[\tilde{\theta},1]$. But we will use \eqref{decreasing-near-1} for convenience.

The main result in the present paper is stated in the following theorem.

\begin{thm}\label{thm-transition-front}
Suppose $\rm(H1)$ and $\rm(H2)$. Equation \eqref{main-eqn} admits a transition front $u(t,x)$ that is strictly decreasing in space and uniformly Lipschitz continuous in space, that is,
\begin{equation*}
\sup_{x\neq y, t\in\R}\bigg|\frac{u(t,y)-u(t,x)}{y-x}\bigg|<\infty.
\end{equation*}
Moreover, there exists a continuous differentiable function $X:\R\ra\R$ such that the following hold:
\begin{itemize}
\item[\rm(i)] there exist $c_{\min}^*>0$ and $c_{\max}^*>0$ such that $c_{\min}^*\leq\dot{X}(t)\leq c_{\max}^*$ for all $t\in\R$;

\item[\rm(ii)] there exist two exponents $c_{\pm}>0$ and two shifts $h_{\pm}>0$ such that
\begin{equation*}
\begin{split}
u(t,x+X(t)+h_{+})&\leq e^{-c_{+}x},\quad x\geq0,\\
u(t,x+X(t)-h_{-})&\geq 1-e^{c_{-}x},\quad x\leq0
\end{split}
\end{equation*}
for all $t\in\R$.
\end{itemize}
\end{thm}

Clearly, due to space homogeneity, any space translation of $u(t,x)$ is also a transition front. Considering space reflection, we obtain space increasing transition fronts. Note that if $u(t,x)$ is continuously differentiable, then $\phi(t,x):=u(t,x+X(t))$ satisfies
\begin{equation*}
\begin{cases}
\phi_{t}=J\ast\phi-\phi+\dot{X}(t)\phi_{x}+f(t,\phi),\\
\phi(t,-\infty)=1,\,\,\phi(t,\infty)=0\,\,\text{uniformly in}\,\,t\in\R.
\end{cases}
\end{equation*}
In the literature, $u(t,x)=\phi(t,x-X(t))$ is also called a \textit{generalized traveling wave}, especially in the time heterogeneous media (see e.g. \cite{NaRo12,Sh99-1,Sh99-2,Sh11}).

Theorem \ref{thm-transition-front} is a nonlocal version of \cite[Theorem 1.3(1)]{ShSh14}, where the classical random dispersal case is treated. However, its proof does not follow the procedure for that of \cite[Theorem 1.3(1)]{ShSh14} due to the lack of space regularity. For clarity, let us outline the main ideas of the proof of Theorem \ref{thm-transition-front} within the following four steps.

\begin{itemize}
\item[\rm(i)] We first construct space decreasing front-like solutions $u(t,x;s)$ of \eqref{main-eqn} with initial moment at $s<0$ satisfying the normalization $u(0,0;s)=\theta$ (see \eqref{front-like-eq},  Lemmas \ref{lem-approximating-sol} and Lemma \ref{lem-approximating-sol-prop}). We are going to show that $u(t,x;s)$ converge to a transition front as $s\to-\infty$. Unlike in the classical random dispersal case, right now, we cannot conclude the convergence of $u(t,x;s)$ to some entire solution of \eqref{main-eqn} due to the lack of space regularity. In fact, from continuity and monotonicity, we only know that they are a.e. differentiable in space. Thus, in order for the convergence, certain space regularity needs to be established. Also, in order for $u(t,x;s)$ to converge to a transition front, the uniform boundedness of the interface width of $u(t,x;s)$ needs to be established.

\item[\rm(ii)] For $\la\in(0,1)$, define the interface location $X_{\la}(t;s)$ via $u(t,X_{\la}(t;s);s)=\la$. One important step in constructing transition fronts is to establish the boundedness of interface width, that is,
\begin{equation}\label{intro-interface-width}
\sup_{s<0,t\geq s}\big[X_{\la_{1}}(t;s)-X_{\la_{2}}(t;s)\big]<\infty
\end{equation}
for $0<\la_{1}<\la_{2}<1$. Combining the rightward propagation estimate of $X_{\la}(t;s)$ and an idea of Zlato\v{s} (see \cite[Lemma 2.5]{Zl13}), we are able to show \eqref{intro-interface-width} for $0<\la_{1}<\la_{2}\leq\la_{*}$, where $\la_{*}\in(\theta,1)$ is fixed
(see Theorem \ref{thm-bounded-interface-width}). Of course, this is not enough to describe the front-like shape of $u(t,x;s)$. We need \eqref{intro-interface-width} to be true for all $0<\la_{1}<\la_{2}<1$.

\item[\rm(iii)] Another important step in constructing transition fronts is to establish the steepness estimate, that is, $u_{x}(t,X_{\theta}(t;s);s)$ is uniformly negative, since then the formula
\begin{equation*}
\dot{X}_{\theta}(t;s)=-\frac{u_{t}(t,X_{\theta}(t;s);s)}{u_{x}(t,X_{\theta}(t;s);s)}
\end{equation*}
ensures the finite speed of $X_{\theta}(t;s)$. However, this does not work here for that $u_{x}(t,X_{\theta}(t;s);s)$ may not exist due to the lack of space regularity. To circumvent this, we look at the problem from a different viewpoint. Instead of trying to obtain some properties of $X_{\theta}(t;s)$, we modify it to get what we want. As a result, we obtain a new interface location $X(t;s)$, which is continuously differentiable, has bounded and uniform positive derivative and stays within a neighborhood of $X_{\theta}(t;s)$, and hence, captures the propagation nature of $u(t,x;s)$ (see Theorem \ref{thm-interface-rightward-propagation-improved}). We then use this new interface location to derive the uniform exponential decaying estimate of $u(t,x;s)$ at $\infty$ and of $1-u(t,x;s)$ at $-\infty$ (see Theorem \ref{thm-decaying-estimate}). From which, we establish \eqref{intro-interface-width} for all $0<\la_{1}<\la_{2}<1$ (see Corollary \ref{cor-bound-interface-width}).

\item[\rm(iv)] We have obtained all the expected properties of $u(t,x;s)$ except one: certain space regularity. Here, we establish the uniform Lipschitz continuity in space, that is,
\begin{equation*}
\sup_{x\in\R,\eta\neq0\atop s<0,t\geq s}\bigg|\frac{u(t,x+\eta;s)-u(t,x;s)}{\eta}\bigg|<\infty
\end{equation*}
(see Lemma \ref{lem-reguarity}).
The idea of proof is based on the decomposition of $\R$ into moving intervals according to the new interface location in $\rm(iii)$. Then, for any fixed $x$, it can stay in the interval causing the growth of $\big|\frac{u(t,x+\eta;s)-u(t,x;s)}{\eta}\big|$ only for a certain period of time. For the analysis, we need \eqref{intro-interface-width}. With this regularity and the above-mentioned properties, we finally construct a transition front, and finish the proof of Theorem \ref{thm-transition-front}.
\end{itemize}

Although the interface locations $X(t)$ in Theorem \ref{thm-transition-front} have very nice properties, in general, we have no idea about the value $u(t,X(t))$. To have a better understanding of how $u(t,x)$ propagates, let us look at the interface locations at the ignition temperature, that is, $X_{\theta}(t)$ such that $u(t,X_{\theta}(t))=\theta$ for all $t\in\R$. Due to the space monotonicity of $u(t,x)$, $X_{\theta}(t)$ is well-defined and continuous. We prove

\begin{thm}\label{thm-bounded-oscillation}
Suppose $\rm(H1)$ and $\rm(H2)$. Then, for any $\de>0$, there holds
\begin{equation*}
\sup_{|t-s|\leq\de}|X_{\theta}(t)-X_{\theta}(s)|<\infty.
\end{equation*}
In particular, $\limsup_{t\to\pm\infty}\big|\frac{X_{\theta}(t)}{t}\big|<\infty$.
\end{thm}

It is known that $X_{\theta}(t)$ oscillates due to the time dependence of $f$. Theorem \ref{thm-bounded-oscillation} then says that such oscillation is locally uniformly bounded.

The paper is organized as follows. In Section \ref{sec-approximation-sol}, we construct approximating solutions and study some fundamental properties. In Section \ref{sec-bound-interface-width}, we show that interface width of approximating solutions remains uniformly bounded as time elapses. In Section \ref{sec-modified-decaying}, we construct modified interface location and use it to derive uniform exponential decaying estimates of approximating solutions. In the final section, Section \ref{sec-transition-front}, we study the uniform Lipschitz continuity of approximating solutions and finish the proof of Theorem \ref{thm-transition-front} and Theorem \ref{thm-bounded-oscillation}. The paper is ended up with Appendix \ref{sec-app-CP} on comparison principles.


\section{Approximating front-like  solutions}\label{sec-approximation-sol}

In this section, we construct approximating front-like solutions of \eqref{main-eqn}, which will be shown to converge to a
transition solution of \eqref{main-eqn}. Throughout this section, we assume $\rm(H1)$ and $\rm(H2)$. For the comparison principle, it is referred to Proposition \ref{lem-comparison} in the appendix.

First we note that, by general semigroup theory (see e.g. \cite{Paz}),  for any $u_0\in C_{\rm unif}^b(\R,\R)$ and $s\in\R$,  \eqref{main-eqn}
has a unique (local) solution $u(t,\cdot;s,u_0)\in C_{\rm unif}^b(\R,\R)$ with $u(s,x;s,u_0)=u_0(x)$, where
$$
C_{\rm unif}^b(\R,\R)=\{u\in C(\R,\R)\,|\, u\,\, \text{is uniformly continuous on}\,\,\R\,\,\text{and}\,\, \sup_{x\in\R}|u(x)|<\infty\}
$$
equipped with the norm $\|u\|=\sup_{x\in\R}|u(x)|$. Moreover, $u(t,\cdot;s,u_0)$ is continuous in $s\in\R$ and $u_0\in C_{\rm unif}^b(\R,\R)$.
By the comparison principle, if $u_0(x)\ge 0$ for $x\in\R$, then
$u(t,\cdot;s,u_0)$ exists for all $t\ge s$ and $u(t,x;s,u_0)\ge 0$ for $t\ge s$ and $x\in\R$.

Consider the following homogeneous equation
\begin{equation}\label{main-eqn-perturb-homo}
u_{t}=J\ast u-u+f_{\min}(u),\quad (t,x)\in\R\times\R,
\end{equation}
where $f_{\min}$, given in $\rm(H2)$, is of ignition type. Let us first summarize some results obtained in \cite{Cov-thesis}. There are a unique $c_{\min}>0$ and a unique continuously differentiable function $\phi=\phi_{\min}:\R\to(0,1)$ satisfying
\begin{equation}\label{tw-homo}
\begin{cases}
J\ast\phi-\phi+c_{\min}\phi_{x}+f_{\min}(\phi)=0,\\
\phi_{x}<0,\,\,\phi(0)=\theta,\,\,\phi(-\infty)=1\,\,\text{and}\,\,\phi(\infty)=0.
\end{cases}
\end{equation}
That is, $\phi_{\min}$ is the normalized wave profile and $\phi_{\min}(x-c_{\min}t)$ is the traveling wave of \eqref{main-eqn-perturb-homo}. Moreover, $\phi_{\min}\in C^{1}(\R)$ and it enjoys exponential decaying estimates stated in the following lemma.

\begin{lem}\label{lem-app-exponential-decay}
There exist $c_{\min}^{\pm}>0$ and $x_{\min}<0$ such that
\begin{equation*}
\begin{split}
\phi_{\min}(x)&\leq e^{-c_{\min}^{+}x},\quad x\geq0,\\
1-\phi_{\min}(x)&\leq e^{c_{\min}^{-}(x-x_{\min})},\quad x\leq x_{\min}.
\end{split}
\end{equation*}
\end{lem}
We remark that it is shown in \cite{Cov-thesis} that $\phi_{\min}$ does not decay faster than exponential function at $\infty$ and $1-\phi_{\min}$ does not decay faster than exponential function at $-\infty$.

For $s<0$ and $y\in\R$, denote by $u(t,x;s,\phi_{\min}(\cdot-y))$ the classical solution of \eqref{main-eqn} with initial data $u(s,x;s,\phi_{\min}(\cdot-y))=\phi_{\min}(x-y)$. We remark that there is no guarantee of space regularity of $u(t,x;s,\phi_{\min}(\cdot-y))$. The next lemma gives the approximating solutions.

\begin{lem}\label{lem-approximating-sol}
For any $s<0$, there exists a unique $y_{s}\in\R$ such that $u(0,0;s,\phi_{\min}(\cdot-y_{s}))=\theta$. Moreover, $y_{s}\to-\infty$ as $s\to-\infty$.
\end{lem}
\begin{proof}
Let $s<0$. We first see that the comparison principle gives
\begin{equation*}
u(t,x;s,\phi_{\min}(\cdot-y))\geq\phi_{\min}(x-y-c_{\min}(t-s)),\quad x\in\R,\,\,t\geq s.
\end{equation*}
In particular, $u(0,0;s,\phi_{\min}(\cdot-y))\geq\phi_{\min}(c_{\min}s-y)$. Thus, the monotonicity and the normalization of $\phi_{\min}$ ensure that $u(0,0;s,\phi_{\min}(\cdot-y))\geq\theta$ if $y\geq c_{\min}s$.

We now construct a suitable supersolution to bound $u(t,x;s,\phi_{\min}(\cdot-y))$ from above. By Lemma \ref{lem-app-exponential-decay}, we have $\phi_{\min}(\cdot-y)\leq e^{-c_{\min}^{+}(\cdot-y)}$ for all $y\in\R$. Set $v^{y}(t,\cdot;s)=e^{-c_{\min}^{+}(\cdot-y-c(t-s))}$, where $c>0$ is to be chosen. We compute
\begin{equation*}
v^{y}_{t}-[J\ast v^{y}-v^{y}]=\bigg[c_{\min}^{+}c-\int_{\R}J(z)e^{c_{\min}^{+}z}dz+1\bigg]v^{y}\geq f(t,v^{y})
\end{equation*}
provided $c>0$ is so large that $c_{\min}^{+}c-\int_{\R}J(z)e^{c_{\min}^{+}z}dz+1\geq\sup_{u\in(0,1)}\frac{f_{\max}(u)}{u}$. It then follows from the comparison principle that
\begin{equation*}
u(t,x;s,\phi_{\min}(\cdot-y))\leq e^{-c_{\min}^{+}(x-y-c(t-s))},\quad x\in\R,\,\,t\geq s.
\end{equation*}
In particular, $u(0,0;s,\phi_{\min}(\cdot-y))\leq\theta$ if $y\leq\frac{\ln\theta+c_{\min}^{+}cs}{c_{\min}^{+}}$.

Continuity of $u(0,0;s,\phi_{\min}(\cdot-y))$ in $y$ then yields the existence of some $y_{s}\in\R$ such that $u(0,0;s,\phi_{\min}(\cdot-y_{s}))=\theta$. The uniqueness of such an $y_{s}$ is a simple consequence of the comparison principle. In fact, if $y_{1}<y_{2}$, then $\phi_{\min}(\cdot-y_{1})<\phi_{\min}(\cdot-y_{2})$ by monotonicity, which implies that $u(0,0;s,\phi_{\min}(\cdot-y_{1}))<u(0,0;s,\phi_{\min}(\cdot-y_{2}))$.

Moreover, the above analysis implies that
\begin{equation}\label{initial-position}
\frac{\ln\theta+c_{\min}^{+}cs}{c_{\min}^{+}}\leq y_{s}\leq c_{\min}s,
\end{equation}
and hence, $y_{s}\to-\infty$ as $s\to-\infty$.
\end{proof}

For notational simplicity, in what follows, we put
\begin{equation}\label{front-like-eq} 
u(t,x;s)=u(t,x;s,\phi_{\min}(\cdot-y_{s})).
\end{equation}
Thus, $u(s,\cdot;s)=\phi_{\min}(\cdot-y_{s})$. The next lemma provides some fundamental results of $u(t,x;s)$.

\begin{lem}\label{lem-approximating-sol-prop}
For any $s<0$ and $t\geq s$ there hold
\begin{itemize}
\item[\rm(i)] the limits $u(t,-\infty;s)=1$ and $u(t,\infty;s)=0$ are locally uniformly in $s$ and $t$;
\item[\rm(ii)] $u(t,x;s)$ is strictly decreasing in $x$. In particular, $u(t,x;s)$ is almost everywhere differentiable in $x$.
\end{itemize}
\end{lem}
\begin{proof}
$\rm(i)$ It follows from the fact $u(t,x;s)\in(0,1)$ by the comparison principle, the estimate \eqref{initial-position} and the following estimate
\begin{equation}\label{aprior-estimate-two-side1}
\phi(x-y_{s}-c_{\min}(t-s))\leq u(t,x;s)\leq e^{-c_{\min}^+(x-y_{s}-c(t-s))}
\end{equation}
for some sufficiently large $c>0$, which is derived in Lemma \ref{lem-approximating-sol}.

$\rm(ii)$ For the monotonicity, we first see that $u(s,x;s)=\phi_{\min}(x-y_{s})$ is strictly decreasing in $x$. For any $y>0$, we apply comparison principle to $u(t,x-y;s)-u(t,x;s)$ to conclude that that $u(t,x-y;s)>u(t,x;s)$ for $t>s$. The result then follows.
\end{proof}

We remark that the family $\{u(t,x;s)\}_{s<0}$ is the approximating front-like solutions, which will be shown to converge to a transition front of \eqref{main-eqn} as $s\to -\infty$. However, due to the lack of space regularity as mentioned before, it is not clear that this family will converge to some solution of \eqref{main-eqn}.  Also even if $u(t,x;s)$ converges to some solution of \eqref{main-eqn}, it is difficult to see that the limiting solution is a transition front. We will then first establish in Sections \ref{sec-bound-interface-width} and \ref{sec-modified-decaying} the uniform boundedness of the interface width of $u(t,x;s)$ and uniform decaying estimates of $u(t,x;s)$, respectively, which assure the limiting solution of $u(t,x;s)$ as $s\to -\infty$ (if exists) is a transition front. Later, in Lemma \ref{lem-reguarity}, we will show the uniform Lipschitz continuity in space of the approximating solutions, which of course implies the convergence of the approximating solutions thanks to Arzel\`{a}-Ascoli theorem, but its proof, using the uniform boundedness of interface width (see Corollary \ref{cor-bound-interface-width}), is not straightforward.


\section{Bounded interface width}\label{sec-bound-interface-width}

For $s<0$, $t\geq s$ and $\la\in(0,1)$, let $X_{\la}(t;s)$ be such that $u(t,X_{\la}(t;s);s)=\la$. By Lemma \ref{lem-approximating-sol-prop}, it is well-defined and continuous in $t$. Moreover, $X_{\la_{1}}(t;s)>X_{\la_{2}}(t;s)$ if $\la_{1}<\la_{2}$. We point out that $X_{\la}(t;s)$ is not monotonic in $t$ due to the time dependence of $f$. As usual, we refer to $X_{\la}(t;s)$ as the interface location, and the corresponding point on the solution curve as the interface. Assume $\rm(H1)$ and $\rm(H2)$ in this section. The main result in this section is stated in the following theorem.

\begin{thm}\label{thm-bounded-interface-width}
There exists $\la_{*}\in(\theta,1)$ such that for any $0<\la_{1}<\la_{2}\leq\la_{*}$, there holds
\begin{equation*}
\sup_{s<0,t\geq s}\big[X_{\la_{1}}(t;s)-X_{\la_{2}}(t;s)\big]<\infty.
\end{equation*}
\end{thm}

Theorem \ref{thm-bounded-interface-width} shows the boundedness of the width between interfaces below $\la_{*}$. Later in Corollary \ref{cor-bound-interface-width}, it is extended to any $0<\la_{1}<\la_{2}<1$. The proof of Theorem \ref{thm-bounded-interface-width} is long, and therefore, broken into several parts. In Subsection \ref{subsec-progation-estimate}, we give rightward propagation estimates of interfaces above the ignition temperature. In Subsection \ref{subsec-aux-width}, we define new interface locations $Y_{\ka}(t;s)$ and establish the uniform boundedness between $Y_{\ka}(t;s)$ and $X_{\la}(t;s)$. We finish the proof of Theorem \ref{thm-bounded-interface-width} in Subsection \ref{subsec-proof}.


\subsection{Rightward propagation estimates of interface locations}\label{subsec-progation-estimate}

In this subsection, we study the rightward propagation nature of $X_{\la}(t;s)$. To do so, we need some knowledge on the traveling waves of
\begin{equation}\label{eqn-homo-bistable}
u_{t}=J\ast u-u+f_{B}(u)
\end{equation}
where $f_{B}:[0,1]\to\R$ is a bistable nonlinearity satisfying the following conditions
\begin{equation*}
\begin{cases}
f_{B}\in C^{2}([0,1]),\,\,f_{B}(0)=0,\,\,f_{B}(\theta)=0,\,\,f_{B}(1)=0,\\
f_{B}'(0)<0,\,\,f_{B}'(1)<0,\\
f_{B}(u)<0\,\,\text{for}\,\,u\in(0,\theta),\,\,0<f_{B}(u)\leq f_{\min}(u)\,\,\text{for}\,\,u\in(\theta,1),\\
\int_{0}^{1}f_{B}(u)du>0\,\,\text{and}\,\,1+f_{B}'(u)>0\,\,\text{for}\,\,u\in[0,1].
\end{cases}
\end{equation*}
It is proven in \cite{BaFiReWa97} that the problem
\begin{equation}\label{bistable-tw-homo}
\begin{cases}
J\ast u-u+cu_{x}+f_{B}(u)=0,\\
u_{x}<0,\,\, u(0)=\theta,\,\, u(-\infty)=1\,\,\text{and}\quad u(\infty)=0
\end{cases}
\end{equation}
admits a unique solution $(c_{B},\phi_{B})$ with $c_{B}>0$. Moreover, the following stability result holds:

\begin{lem}\label{lem-app-uniform-stable-bi}
Let $u_{0}:\R\to[0,1]$ be continuous and satisfy
\begin{equation*}
\limsup_{x\to\infty}u_{0}(x)<\theta<\liminf_{x\to-\infty}u_{0}(x).
\end{equation*}
Then, there exists $x_{B}=x_{B}(u_{0})\in\R$, $q_{B}=q_{B}(u_{0})>0$ and $\om_{B}>0$ such that
\begin{equation*}
u_{B}(t,x;u_{0})\geq\phi_{B}(x-x_{B}-c_{B}t)-q_{B}e^{-\om_{B}t},\quad x\in\R
\end{equation*}
for all $t\geq0$, where $u_{B}(t,x;u_{0})$ is the solution of \eqref{eqn-homo-bistable} with initial data $u_{B}(0,\cdot;u_{0})=u_{0}$.
\end{lem}

The main result in this subsection is stated in the following proposition.

\begin{prop}\label{prop-rightward-propagation}
Let $\la\in(\theta,1)$. For any $\ep>0$ there exists $t_{\ep,\la}>0$ such that
\begin{equation*}
X_{\la}(t;s)-X_{\la}(t_{0};s)\geq(c_{B}-\ep)(t-t_{0}-t_{\ep,\la})
\end{equation*}
for all $s<0$, $t\geq t_{0}\geq s$.
\end{prop}
\begin{proof}
Fix some $\la\in(\theta,1)$. Let $u_{0}:\R\to[0,1]$ be a uniformly continuous and nonincreasing function satisfying $u_{0}(x)=\la$ for $x\leq x_{0}$ and $u_{0}(x)=0$ for $x\geq0$, where $x_{0}<0$ is fixed.
Clearly, space monotonicity of $u(t,x;s)$ implies that
\begin{equation*}
u(t_{0},x+X_{\la}(t_{0};s);s)\geq u_{0}(x),\quad x\in\R,\,\,t_{0}\geq s,
\end{equation*}
and then, by $f(t,u)\geq f_{\min}(u)\geq f_{B}(u)$ and the comparison principle, we find
\begin{equation*}
u(t,x+X_{\la}(t_{0};s);s)\geq u_{B}(t-t_{0},x;u_{0}),\quad x\in\R,\,\,t\geq t_{0}\geq s.
\end{equation*}
By Lemma \ref{lem-app-uniform-stable-bi}, there are constants $x_{B}=x_{B}(\la)\in\R$, $q_{B}=q_{B}(\la)>0$ and $\om_{B}>0$ such that
\begin{equation*}
u_{B}(t-t_{0},x;u_{0})\geq\phi_{B}(x-x_{B}-c_{B}(t-t_{0}))-q_{B}e^{-\om_{B}(t-t_{0})}, \quad x\in\R,\,\,t\geq t_{0}\geq s.
\end{equation*}
Hence,
\begin{equation*}
u(t,x+X_{\la}(t_{0};s);s)\geq \phi_{B}(x-x_{B}-c_{B}(t-t_{0}))-q_{B}e^{-\om_{B}(t-t_{0})}, \quad x\in\R,\,\,t\geq t_{0}\geq s.
\end{equation*}
Let $T_{0}=T_{0}(\la)$ be such that $q_{B}e^{-\om_{B}T_{0}}=\frac{1-\la}{2}$ and denote by $\xi_{B}(\frac{1+\la}{2})$ the unique point such that $\phi_{B}(\xi_{B}(\frac{1+\la}{2}))=\frac{1+\la}{2}$. Setting $x=x_{B}+c_{B}(t-t_{0})+\xi_{B}(\frac{1+\la}{2})$, we find for $t\geq t_{0}+T_{0}$
\begin{equation*}
u(t,x_{B}+c_{B}(t-t_{0})+\xi_{B}(\frac{1+\la}{2})+X_{\la}(t_{0};s);s)\geq \phi_{B}(\xi_{B}(\frac{1+\la}{2}))-q_{B}e^{-\om_{B}T_{0}}=\la,
\end{equation*}
which together with monotonicity implies that
\begin{equation}\label{estimate-long-time}
X_{\la}(t;s)-X_{\la}(t_{0};s)\geq x_{B}+c_{B}(t-t_{0})+\xi_{B}(\frac{1+\la}{2}),\quad t\geq t_{0}+T_{0}.
\end{equation}

We now estimate $X_{\la}(t;s)-X_{\la}(t_{0};s)$ for $t\in[t_{0},t_{0}+T_{0}]$. We claim that there exists $z=z(T_{0})<0$ such that
\begin{equation}\label{estimate-finite-time}
X_{\la}(t;s)-X_{\la}(t_{0};s)\geq z\quad\text{for all}\,\,s<0,\,\,s\leq t_{0}\leq t\leq t_{0}+T_{0}.
\end{equation}
Let $u_{\min}(t,x;u_{0})$ and $u_{\min}(t;\la):=u_{\min}(t,x;\la)$ be solutions of \eqref{main-eqn-perturb-homo} with $u_{\min}(0,x;u_{0})=u_{0}(x)$ and $u_{\min}(0;\la)=u_{\min}(0,x;\la)\equiv\la$, respectively. By the comparison principle, we have $u_{\min}(t,x;u_{0})<u_{\min}(t;\la)$ for all $x\in\R$ and $t>0$, and $u_{\min}(t,x;u_{0})$ is strictly decreasing in $x$ for $t>0$.

We see that for any $t>0$, $\lim_{x\to-\infty}u_{\min}(t,x;u_{0})=u_{\min}(t;\la)$. This is because that $\frac{d}{dt}u_{\min}(t,-\infty;u_{0})=f_{\min}(u_{\min}(t,-\infty;u_{0}))$ for $t>0$ and $u_{\min}(0,-\infty;u_{0})=\la$. Since $\la\in(\theta,1)$, as a solution of the ODE $u_{t}=f_{\min}(u)$, $u_{\min}(t;\la)$ is strictly increasing in $t$, which implies that $u_{\min}(t,-\infty;u_{0})=u_{\min}(t;\la)>\la$ for $t>0$. As a result, for any $t>0$ there exists a unique $\xi_{\min}(t)\in\R$ such that $u_{\min}(t,\xi_{\min}(t);u_{0})=\la$. Moreover, $\xi_{\min}(t)$ is continuous in $t$.

Since $f(t,u)\geq f_{\min}(u)$ and $u(t_{0},\cdot+X_{\la}(t_{0};s);s)\geq u_{0}$, the comparison principle implies that
\begin{equation*}
u(t,x+X_{\la}(t_{0};s);s)\geq u_{\min}(t-t_{0},x;u_{0}),\quad x\in\R,\,\,t\geq t_{0}\geq s.
\end{equation*}
Setting $x=\xi_{\min}(t-t_{0})$, we find $u(t,\xi_{\min}(t-t_{0})+X_{\la}(t_{0};s);s)\geq\la$, which together with the monotonicity implies that
$X_{\la}(t;s)\geq\xi_{\min}(t-t_{0})+X_{\la}(t_{0};s)$ for $t\geq t_{0}\geq s$.
Thus, \eqref{estimate-finite-time} follows if $\inf_{t\in(t_{0},t_{0}+T_{0}]}\xi_{\min}(t-t_{0})>-\infty$, that is,
\begin{equation}\label{not-to-negative-finity}
\inf_{t\in(0,T_{0}]}\xi_{\min}(t)>-\infty.
\end{equation}

We now show \eqref{not-to-negative-finity}. Since $u_{0}(x)=\la$ for $x\leq x_{0}$, continuity with respect to the initial data
 (in sup norm) implies that for any $\ep>0$ there exists $\de>0$ such that
\begin{equation*}
u_{\min}(t;\la)-\la\leq\ep\quad\text{and}\quad\sup_{x\leq x_{0}}[u_{\min}(t;\la)-u_{\min}(t,x;u_{0})]=u_{\min}(t;\la)-u_{\min}(t,x_{0};u_{0})\leq\ep
\end{equation*}
for all $t\in[0,\de]$, where the equality is due to monotonicity. By $\rm(H1)$, $J$ concentrates near $0$ and decays very fast as $x\to\pm\infty$. Thus, we can choose $x_{1}=x_{1}(\ep)<<x_{0}$ such that $\int_{-\infty}^{x_{0}}J(x-y)dy\geq1-\ep$ for all
$x\le x_1$. Now, for any $x\leq x_{1}$ and $t\in(0,\de]$, we have
\begin{equation*}
\begin{split}
\frac{d}{dt}u_{\min}(t,x;u_{0})&=\int_{\R}J(x-y)u_{\min}(t,y;u_{0})dy-u_{\min}(t,x;u_{0})+f_{\min}(u_{\min}(t,x;u_{0}))\\
&\geq\int_{-\infty}^{x_{0}}J(x-y)u_{\min}(t,y;u_{0})dy-u_{\min}(t,x;u_{0})+f_{\min}(u_{\min}(t,x;u_{0}))\\
&\geq(1-\ep)\inf_{x\leq x_{0}}u_{\min}(t,x;u_{0})-u_{\min}(t;\la)+f_{\min}(u_{\min}(t,x;u_{0}))\\
&=-(1-\ep)\sup_{x\leq x_{0}}[u_{\min}(t;\la)-u_{\min}(t,x;u_{0})]-\ep u_{\min}(t;\la)+f_{\min}(u_{\min}(t,x;u_{0}))\\
&\geq-\ep(1-\ep)-\ep(\la+\ep)+f_{\min}(u_{\min}(t,x;u_{0}))\\
&>0
\end{split}
\end{equation*}
if we choose $\ep>0$ sufficiently small, since then $f_{\min}(u_{\min}(t,x;u_{0}))$ is close to $f_{\min}(\la)$, which is positive. This simply means that $u_{\min}(t,x;u_{0})>\la$ for all $x\leq x_{1}$ and $t\in(0,\de]$, which implies that $\xi_{\min}(t)>x_{1}$ for $t\in(0,\de]$. The continuity of $\xi_{\min}$ then leads to \eqref{not-to-negative-finity}. This proves \eqref{estimate-finite-time}. The result of the proposition then follows from \eqref{estimate-long-time} and \eqref{estimate-finite-time}.
\end{proof}


\subsection{Auxiliary interface width estimates}\label{subsec-aux-width}

In this subsection, we define a new interface location and study the boundedness between this new interface location and the original ones.

For $\ka>0$, set $c_{\ka}:=\inf_{\la>0}\frac{1}{\la}\big(\int_{\R}J(y)e^{\la y}dy-1+\ka\big)>0$. It is not hard to see that there exists a unique $\la_{\ka}>0$ such that
\begin{equation}\label{minimal-point}
c_{\ka}=\frac{1}{\la_{\ka}}\bigg(\int_{\R}J(y)e^{\la_{\ka}y}dy-1+\ka\bigg).
\end{equation}
We point out (see \cite{CaCh04}) that if the support of $J(\cdot)$ is compact, then  $c_{\ka}$ is  the minimal speed, denoted by $c_{\rm KPP}$,  of the traveling waves of the following KPP equation,
\begin{equation*}
u_{t}=J\ast u-u+f_{\rm KPP}(u)
\end{equation*}
with $f_{\rm KPP}'(0)=\ka$, where $f_{\rm KPP}(0)=0$ and $f_{\rm KPP}(u)/u$ is strictly decreasing in $u\ge 0$.  If $J$ is not compactly supported, it is known that $c_{\rm KPP}\leq c_\ka$ (see \cite{CoDu05}). As in the classical random dispersal case, we have

\begin{lem}\label{lem-aux-minimal-speed}
$c_{\ka}\to0$ and $\la_{\ka}\to0$ as $\ka\ra0$.
\end{lem}
\begin{proof}
We see
\begin{equation*}
c_{\ka}\leq\frac{1}{\sqrt{\ka}}\bigg(\int_{\R}J(y)e^{\sqrt{\ka}y}dy-1+\ka\bigg)\to0\quad\text{as}\quad\ka\to0,
\end{equation*}
since $\frac{1}{\sqrt{\ka}}\big(\int_{\R}J(y)e^{\sqrt{\ka}y}dy-1\big)\to0$ as $\ka\to0$.

We show $\la_{\ka}\to0$ as $\ka\ra0$. It is understood that $\la_{\ka}$ is the unique positive point such that $\frac{d}{d\la}\big[\frac{1}{\la}\big(\int_{\R}J(y)e^{\la y}dy-1+\ka\big)\big]=0$, that is,
\begin{equation*}
\la\int_{\R}J(y)ye^{\la y}dy-\int_{\R}J(y)e^{\la y}dy+1=\ka.
\end{equation*}
Setting $g(\la):=\la\int_{\R}J(y)ye^{\la y}dy-\int_{\R}J(y)e^{\la y}dy+1$, we see $g(0)=0$ and
\begin{equation*}
g'(\la)=\la\int_{\R}J(y)y^{2}e^{\la y}dy>0\quad\text{for}\,\,\la>0.
\end{equation*}
This simply means that the unique solution of $g(\la)=\ka$ goes to $0$ as $\ka\to0$. This completes the proof.
\end{proof}

For $\ka>0$, $s<0$ and $t\geq s$, define
\begin{equation}\label{new-interface}
Y_{\ka}(t;s)=\inf\Big\{y\in\R\Big|u(t,x;s)\leq e^{-\la_{\ka}(x-y)},\quad x\in\R\Big\}.
\end{equation}
Due to the second estimate in \eqref{aprior-estimate-two-side1}, $Y_{\ka}(t;s)$ is well-defined if $\la_{\ka}\leq c_{\min}^{+}$, which is the case if $\ka$ is sufficiently small due to Lemma \ref{lem-aux-minimal-speed}. Notice the definition does not guarantee any continuity or monotonicity of $Y_{\ka}(t;s)$ since $u(t,x;s)$ is not monotone in time. The following result controls the propagation of $Y_{\ka}(t;s)$.

\begin{lem}\label{lem-new-interface-bd}
Let $\ka_{0}=\sup_{u\in(0,1)}\frac{f_{\max}(u)}{u}$. For $\ka>0$, set  $\tilde{c}_{\ka}:=\frac{1}{\la_{\ka}}\big(\int_{\R}J(y)e^{\la_{\ka}y}dy-1+\ka_{0}\big)$, where $\la_{\ka}$ is given in \eqref{minimal-point}. Then, for any small $\ka>0$, we have $\tilde{c}_{\ka}>0$ and
\begin{equation*}
Y_{\ka}(t;s)-Y_{\ka}(t_{0};s)\leq\tilde{c}_{\ka}(t-t_{0})
\end{equation*}
for all $s<0$, $t\geq t_{0}\geq s$.
\end{lem}
\begin{proof}
For small $\ka>0$, we have $\tilde{c}_{\ka}\geq c_{\ka}>0$. For $s<0$, $t\geq t_{0}\geq s$, define
\begin{equation*}
v(t,x;t_{0})=e^{-\la_{\ka}(x-Y_{\ka}(t_{0};s)-\tilde{c}_{\ka}(t-t_{0}))},\quad x\in\R.
\end{equation*}
By the definition of $\tilde{c}_{\ka}$, we readily check that $v_{t}=J\ast v-v+\ka_{0}v$. By the definition of $\ka_{0}$, we have $\ka_{0}v\geq f_{\max}(v)$ for all $v\geq0$. It then follows from $v(t_{0},x;t_{0})=e^{-\la_{\ka}(x-Y_{\ka}(t_{0};s))}\geq u(t_{0},x;s)$ by \eqref{new-interface} and the comparison principle that $v(t,x;t_{0})\geq u(t,x;s)$ for $t\geq t_{0}$, which leads to the result.
\end{proof}

We now prove the main result in this subsection. In what follows, we fix some small $\ka>0$ such that $c_{\ka}<c_{B}$, and for this fixed $\ka$, we write $Y_{\ka}(t;s)$ as $Y(t;s)$, and set
\begin{equation}\label{a-critical-value}
\la_{*}:=\min\big\{u>0\big|f_{\max}(u)=\ka u\big\}\in(\theta,1).
\end{equation}

\begin{prop}\label{prop-bd-width-aux}
For any $\la\in(\theta,\la_{*}]$, there is $C(\la)>0$ such that
\begin{equation*}
|X_{\la}(t;s)-Y(t;s)|\leq C(\la)
\end{equation*}
for all $s<0$, $t\geq s$.
\end{prop}
\begin{proof}
From the definition of $\la_{*}$, we readily see that
\begin{equation}\label{a-condition-bounded-interface}
f(t,u)\leq f_{\max}(u)\leq \ka u,\quad u\in[0,\la_{*}].
\end{equation}
Fix an $\la\in(\theta,\la_{*}]$. Set $C_{0}=\max\{Y(s;s)-X_{\la}(s;s),1\}$. We see that $C_{0}$ is independent of $s<0$. This is because that $u(s,\cdot;s)=\phi_{\min}(\cdot-y_{s})$, and hence, space translations do not change $Y(s;s)-X_{\la}(s;s)$. Clearly, we have the estimate $\sup_{s<0,t\geq s}[X_{\la}(t;s)-Y(t;s)]\leq C$ for some large $C>0$.

Set $\ep=\frac{c_{B}-c_{\ka}}{2}$ and $C_{1}=C_{0}+c_{B}t_{\ep,\la}$, where $t_{\ep,\la}$ is as in Proposition \ref{prop-rightward-propagation}. To finish the proof, we only need to show $\sup_{s<0,t\geq s}[Y(t;s)-X_{\la}(t;s)]\leq C_{1}$.
 Suppose this is not the case, then we can find some $s_{1}<0$ and $t_{1}\geq s_{1}$ such that $Y(t_{1};s_{1})-X_{\la}(t_{1};s_{1})>C_1$. Since $Y(s_{1};s_{1})-X_{\la}(s_{1};s_{1})\leq C_{0}<C_{1}$, there holds $t_{1}>s_{1}$. Let
\begin{equation*}
t_{0}=\sup\big\{t\in[s_{1},t_{1}]\big|Y(t;s_{1})-X_{\la}(t;s_{1})\leq C_{0}\big\}.
\end{equation*}

We claim $Y(t_{0};s_{1})-X_{\la}(t_{0};s_{1})\leq C_{0}$. It is trivial if there are only finitely many $t\in[s_{1},t_{1}]$ such that $Y(t;s_{1})-X_{\la}(t;s_{1})\leq C_{0}$. So we assume there are infinitely many such $t$ and the claim is false. Then, there exists a sequence $\{\tilde{t}_{n}\}_{n\in\N}\subset[s_{1},t_{0})$ such that $Y(\tilde{t}_{n};s_{1})-X_{\la}(\tilde{t}_{n};s_{1})\leq C_{0}$ for $n\in\N$ and $\tilde{t}_{n}\ra t_{0}$ as $n\ra\infty$. Moreover, $Y(t_{0};s_{1})-X_{\la}(t_{0};s_{1})=\tilde C_{1}>C_{0}$. It then follows that for all $n\in\N$
\begin{equation*}
Y(\tilde{t}_{n};s_{1})-X_{\la}(\tilde{t}_{n};s_{1})\leq C_{0}=C_{0}-\tilde C_{1}+Y(t_{0};s_{1})-X_{\la}(t_{0};s_{1}),
\end{equation*}
that is,
\begin{equation*}
\tilde C_{1}-C_{0}+X_{\la}(t_{0};s_{1})-X_{\la}(\tilde{t}_{n};s_{1})\leq Y(t_{0};s_{1})-Y(\tilde{t}_{n};s_{1})\leq\tilde{c}_{\ka}(t_{0}-\tilde{t}_{n}),
\end{equation*}
where the second inequality is due to Lemma \ref{lem-new-interface-bd}. Passing $n\to\infty$, we easily deduce a contradiction from the continuity of $X_{\la}(t;s_{1})$. Hence, the claim is true, that is, $Y(t_{0};s_{1})-X_{\la}(t_{0};s_{1})\leq C_{0}$. It follows that $t_{0}<t_{1}$.

We show
\begin{equation}\label{equality-technical}
Y(t_{0};s_{1})-X_{\la}(t_{0};s_{1})=C_{0}.
\end{equation}
Suppose \eqref{equality-technical} is not true, then we can find some $\de_{0}>0$ such that $Y(t_{0};s_{1})-X_{\la}(t_{0};s_{1})=C_{0}-\de_{0}$. Since $Y(t;s_{1})-X_{\la}(t;s_{1})>C_{0}$ for $t\in(t_{0},t_{1}]$ by the definition of $t_{0}$, we deduce from Lemma \ref{lem-new-interface-bd} that for $t\in(t_{0},t_{1}]$
\begin{equation*}
\begin{split}
C_{0}<Y(t;s_{1})-X_{\la}(t;s_{1})&\leq Y(t_{0};s_{1})+\tilde{c}_{\ka}(t-t_{0})-X_{\la}(t_{0};s_{1})+X_{\la}(t_{0};s_{1})-X_{\la}(t;s_{1})\\
&=C_{0}-\de_{0}+\tilde{c}_{\ka}(t-t_{0})+X_{\la}(t_{0};s_{1})-X_{\la}(t;s_{1}).
\end{split}
\end{equation*}
This leads to a contradiction when $t$ approaches $t_{0}$ due to the continuity of $X_{\la}(t;s_{1})$ in $t$. Hence, \eqref{equality-technical} holds.

Next, we look at the time interval $[t_{0},t_{1}]$ and set $\tilde{Y}(t;s_{1})=Y(t_{0};s_{1})+c_{\ka}(t-t_{0})$ for $t\in[t_{0},t_{1}]$. Note both $X_{\la}(t;s_{1})$ and $\tilde{Y}(t;s_{1})$ are continuous, and $X_{\la}(t_{0};s_{1})<\tilde{Y}(t_{0};s_{1})$ by \eqref{equality-technical}. We claim that $X_{\la}(t;s_{1})<\tilde{Y}(t;s_{1})$ for all $t\in[t_{0},t_{1}]$. Suppose this is not the case and let
\begin{equation*}
t_{2}=\min\big\{t\in[t_{0},t_{1}]\big|X_{\la}(t;s_{1})=\tilde{Y}(t;s_{1})\big\}.
\end{equation*}
Clearly, $t_{2}\in(t_{0},t_{1}]$. We define
\begin{equation*}
v(t,x;t_{0})=e^{-\la_{\ka}(x-\tilde{Y}(t;s_{1}))},\quad x\in\R,\,\, t\in[t_{0},t_{2}].
\end{equation*}
We easily check $v_{t}=J\ast v-v+\ka v$. Moreover, we see
\begin{itemize}
\item at the initial moment $t_{0}$, we have $u(t_{0},x;s_{1})\leq e^{-\la_{\ka}(x-Y(t_{0};s_{1}))}=v(t_{0},x;t_{0})$ for $x\in\R$,

\item for $x\leq \tilde{Y}(t;s_{1})$ and $t\in(t_{0},t_{2})$, we have $u(t,x;s_{1})<1\leq v(t,x;t_{0})$,

\item for $x>\tilde{Y}(t;s_{1})$ and $t\in(t_{0},t_{2})$, we have $x>X_{\la}(t;s_{1})$, and hence $u(t,x;s_{1})\leq\la$. As a result, we have
$u_{t}=J\ast u-u+f(t,u)\leq J\ast u-u+\ka u$ by \eqref{a-condition-bounded-interface}.
\end{itemize}
Note, by Lemma \ref{lem-approximating-sol-prop} and the definition of $v(t,x;t_{0})$, the limit $v(t,x;t_{0})-u(t,x;s_{1})\to0$ as $x\to\infty$ is uniform in $t\in[t_{0},t_{2}]$. Then, applying the comparison principle (see Proposition \ref{lem-comparison}) to $v(t,x;t_{0})-u(t,x;s_{1})$, we conclude
\begin{equation*}
u(t,x;s_{1})\leq v(t,x;t_{0})=e^{-\la_{\ka}(x-\tilde{Y}(t;s_{1}))},\quad x\in\R,\,\, t\in[t_{0},t_{2}].
\end{equation*}
It follows that $Y(t;s_{1})\leq\tilde{Y}(t;s_{1})$ for $t\in[t_{0},t_{2}]$ by definition in \eqref{new-interface}. In particular, $Y(t_{2};s_{1})\leq\tilde{Y}(t_{2};s_{1})=X_{\la}(t_{2};s_{1})$. Since $t_{2}\in(t_{0},t_{1}]$, we have $Y(t_{2};s_{1})-X_{\la}(t_{2};s_{1})>C_{0}$ by the definition of $t_{0}$. It is a contradiction. Thus, the claim follows, that is, $X_{\la}(t;s_{1})<\tilde{Y}(t;s_{1})$ for all $t\in[t_{0},t_{1}]$, and repeating the above arguments, we see
\begin{equation}\label{propagation-estimate-refined}
Y(t;s_{1})\leq\tilde{Y}(t;s_{1})=Y(t_{0};s_{1})+c_{\ka}(t-t_{0}),\quad t\in[t_{0},t_{1}].
\end{equation}
It follows from \eqref{propagation-estimate-refined} and  Proposition \ref{prop-rightward-propagation} that for any $t\in[t_{0},t_{1}]$
\begin{equation*}
\begin{split}
Y(t;s_{1})-X_{\la}(t;s_{1})&\leq Y(t_{0};s_{1})+c_{\ka}(t-t_{0})-[X_{\la}(t_{0};s_{1})+(c_{B}-\ep)(t-t_{0}-t_{\ep,\la})]\\
&= C_{0}+(c_{B}-\ep)t_{\ep,\la}-(c_{B}-c_{\ka}-\ep)(t-t_{0})\\
&\leq C_{0}+c_{B}t_{\ep,\la}=C_1.
\end{split}
\end{equation*}
This is a contradiction. Consequently, $Y(t;s)-X_{\la}(t;s)\leq C_{1}$ for all $s<0$, $t\geq s$. This completes the proof.
\end{proof}

We remark that Proposition \ref{prop-bd-width-aux} is a nonlocal version of \cite[Lemma 3.4]{ShSh14}, whose proof is based on the rightward propagation estimate as in Proposition \ref{prop-rightward-propagation} and an idea of Zlato\v{s} (see \cite[Lemma 2.5]{Zl13}).


\subsection{Proof of Theorem \ref{thm-bounded-interface-width}}\label{subsec-proof}

Let $\la_{*}$ be as in \eqref{a-critical-value} and fix any $0<\la_{1}<\la_{2}\leq\la_{*}$. Consider the function
\begin{equation*}
g(t,x;s)=e^{-\la_{\ka}(x-Y(t;s))},\quad x\in\R.
\end{equation*}
Since $g(t,Y(t;s);s)=1$, there exists a unique $x_{1}>0$ (independent of $s<0$, $t\geq s$) such that $g(t,Y(t;s)+x_{1};s)=\la_{1}$. Since $g(t,x;s)=e^{-\la_{\ka}(x-Y(t;s))}\geq u(t,x;s)$ for $x\geq Y(t;s)$, and trivially, $g(t,x;s)>1>u(t,x;s)$ for $x<Y(t;s)$, we have $Y(t;s)+x_{1}\geq X_{\la_{1}}(t;s)$. It then follows that
\begin{equation*}
X_{\la_{1}}(t;s)-X_{\la_{2}}(t;s)\leq Y(t;s)-X_{\la_{2}}(t;s)+x_{1}.
\end{equation*}
The result then follows from Proposition \ref{prop-bd-width-aux}.


\section{Modified interface locations and decaying estimates}\label{sec-modified-decaying}

In the study of the propagation of the solution $u(t,x;s)$, the propagation of the interface location $X_{\la}(t;s)$, more precisely, how fast it moves, plays a crucial role. In the classical random dispersal case, this problem is transferred into the study of uniform steepness, that is, whether $u_{x}(t,X_{\la}(t;s);s)$ is uniformly negative, since there holds the formula
\begin{equation*}
\dot{X}_{\la}(t;s)=-\frac{u_{t}(t,X_{\la}(t;s);s)}{u_{x}(t,X_{\la}(t;s);s)}.
\end{equation*}
Clearly, this approach does not work here since we are lack of space regularity of $u(t,x;s)$. Moreover, we do not know if $X_{\la}(t;s)$ is differentiable in $t$ and it moves back and forth in general. To circumvent these difficulties, we look at the problem from a different viewpoint. Instead of studying $X_{\la}(t;s)$ directly, we modify it to get a new interface location of expected properties, such as moving in one direction with certain speed and staying within a certain distance from $X_{\la}(t;s)$, which captures the propagation nature of $u(t,x;s)$. This is the purpose of Subsection \ref{subsec-modified-interface-location}. As an application of the new interface location, we derive uniform exponential decaying estimates of $u(t,x;s)$ in Subsection \ref{subsec-decaying-estimate}.

\subsection{Modified interface locations}\label{subsec-modified-interface-location}

In this subsection, we modify $X_{\la}(t;s)$ properly and prove the following theorem.

\begin{thm}\label{thm-interface-rightward-propagation-improved}
Let $\la_{*}$ be as in \eqref{a-critical-value}. There are $c_{\max}>0$, $\tilde{c}_{\max}>0$ and $d_{\max}>0$ such that for any $s<0$, there exists a continuously differentiable function $X(t;s):[s,\infty)\ra\R$ satisfying
\begin{equation*}
\begin{split}
\frac{c_{B}}{2}\leq\dot{X}(t;s)\leq c_{\max},&\quad t\geq s,\\
|\ddot{X}(t;s)|\leq\tilde{c}_{\max},&\quad t\geq s
\end{split}
\end{equation*}
such that
\begin{equation*}
0\leq X(t;s)-X_{\la_{*}}(t;s)\leq d_{\max},\quad t\geq s.
\end{equation*}
Moreover, $\{\dot{X}(\cdot,s)\}_{s<0}$ and $\{\ddot{X}(\cdot,s)\}_{s<0}$ are uniformly bounded and uniformly Lipschitz continuous.

In particular, for any $\la\in(0,\la_{*}]$, there exists $d_{\max}(\la)>0$ such that
\begin{equation*}
|X(t;s)-X_{\la}(t;s)|\leq d_{\max}(\la)\quad\text{if}\,\,\la\in(0,\la_{*}]
\end{equation*}
for all $s<0$, $t\geq s$.
\end{thm}

\begin{proof}
By Proposition \ref{prop-rightward-propagation}, there exists $t_{B}>0$ such that
\begin{equation}\label{modify-propagation-estimate-1}
X_{\la_{*}}(t;s)-X_{\la_{*}}(t_{0};s)\geq\frac{3}{4}c_{B}(t-t_{0}-t_{B}), \quad s<0,\,\,t\geq t_{0}\geq s.
\end{equation}
Recall $Y(t;s)$ is $Y_{\ka}(t;s)$ for some fixed small $\ka>0$ and we have
\begin{equation}\label{modify-bounded-width}
C_{0}:=\sup_{s<0,t\geq s}|X_{\la_{*}}(t;s)-Y(t;s)|<\infty
\end{equation}
by Proposition \ref{prop-bd-width-aux}, and
\begin{equation}\label{modify-propagation-estimate-2}
Y(t;s)-Y(t_{0};s)\leq c_{0}(t-t_{0}), \quad s<0,\,\,t\geq t_{0}\geq s
\end{equation}
by Lemma \ref{lem-new-interface-bd}, where $c_{0}=\tilde{c}_{\ka}$ for the fixed $\ka>0$. We interpret that \eqref{modify-propagation-estimate-1}, \eqref{modify-bounded-width} and \eqref{modify-propagation-estimate-2} imply that $X_{\la_{*}}(t;s)$ moves with a uniformly positive and uniformly bounded average speed. This observation is crucial in the following modification.

We modify $X_{\la_{*}}(t;s)$ as follows. At the initial moment $s$, we define
\begin{equation*}
Z_{0}(t;s)=X_{\la_{*}}(s;s)+2C_{0}+1+\frac{c_{B}}{2}(t-s),\quad t\geq s.
\end{equation*}
Clearly, $X_{\la_{*}}(s;s)<Z_{0}(s;s)$. By \eqref{modify-propagation-estimate-1}, $X_{\la_{*}}(t;s)$ will hit $Z_{0}(t;s)$ sometime after $s$. Let $T_1(s)$ be the first time that $X_{\la_{*}}(t;s)$ hits $Z_{0}(t;s)$, that is, $T_1(s)=\min\big\{t\geq s\big|X_{\la_{*}}(t;s)=Z_{0}(t;s)\big\}$. It follows that
\begin{equation*}
X_{\la_{*}}(t;s)<Z_{0}(t;s)\,\,\text{for}\,\,t\in[s,T_1(s))\quad\text{and}\quad X_{\la_{*}}(T_1(s);s)=Z_{0}(T_1(s);s).
\end{equation*}
Moreover, $T_1(s)-s\in\big[\frac{2}{2c_{0}-c_{B}},\frac{4(2C_{0}+1)}{c_{B}}+3t_{B}\big]$, which is a simple result of \eqref{modify-propagation-estimate-1}, \eqref{modify-bounded-width} and \eqref{modify-propagation-estimate-2}.

Now, at the moment $T_1(s)$, we define
\begin{equation*}
Z_{1}(t;T_1(s))=X_{\la_{*}}(T_1(s);s)+2C_{0}+1+\frac{c_{B}}{2}(t-T_1(s)),\quad t\geq T_1(s).
\end{equation*}
Similarly, $X_{\la_{*}}(T_1(s);s)<Z_{1}(T_1(s);T_1(s))$ and $X_{\la_{*}}(t;s)$ will hit $Z_{1}(t;T_1(s))$ sometime after $T_1(s)$. Denote by $T_{2}(s)$ the first time that $X_{\la_{*}}(t;s)$ hits $Z_{1}(t;T_1(s))$. Then,
\begin{equation*}
X_{\la_{*}}(t;s)<Z_{1}(t;T_1(s))\,\,\text{for}\,\,t\in[T_1(s),T_{2}(s))\quad\text{and}\quad X_{\la_{*}}(T_{2}(s);s)=Z_{1}(T_{2}(s);T_1(s)),
\end{equation*}
and $T_{2}(s)-T_1(s)\in\big[\frac{2}{2c_{0}-c_{B}},\frac{4(2C_{0}+1)}{c_{B}}+3t_{B}\big]$ by \eqref{modify-propagation-estimate-1}, \eqref{modify-bounded-width} and \eqref{modify-propagation-estimate-2}.

Repeating the above arguments, we obtain the following: there is a sequence of times $\{T_{n-1}(s)\}_{n\in\N}$ satisfying $T_{0}(s)=s$ and
\begin{equation*}
T_{n}(s)-T_{n-1}(s)\in\bigg[\frac{2}{2c_{0}-c_{B}},\frac{4(2C_{0}+1)}{c_{B}}+3t_{B}\bigg]\quad \text{for all}\,\,n\in\N,
\end{equation*}
and  for any $n\in\N$
\begin{equation*}
\begin{split}
X_{\la_{*}}(t;s)<Z_{n-1}(t;T_{n-1}(s))\,\,\text{for}\,\,t\in[T_{n-1}(s),T_{n}(s))\quad\text{and}\quad X_{\la_{*}}(T_{n}(s);s)=Z_{n-1}(T_{n}(s);T_{n-1}(s)),
\end{split}
\end{equation*}
where
\begin{equation*}
Z_{n-1}(t;T_{n-1}(s))=X_{\la_{*}}(T_{n-1}(s);s)+2C_{0}+1+\frac{c_{B}}{2}(t-T_{n-1}(s)).
\end{equation*}
Moreover, for any $n\in\N$ and $t\in[T_{n-1}(s),T_{n}(s))$
\begin{equation*}
\begin{split}
&Z_{n-1}(t;T_{n-1}(s))-X_{\la_{*}}(t;s)\\
&\quad\quad\leq X_{\la_{*}}(T_{n-1}(s);s)+2C_{0}+1+\frac{c_{B}}{2}(t-T_{n-1}(s))\\
&\quad\quad\quad-\bigg[X_{\la_{*}}(T_{n-1}(s);s)+\frac{3}{4}c_{B}(t-T_{n-1}(s)-t_{B})\bigg]\\
&\quad\quad=2C_{0}+1+\frac{3}{4}c_{B}t_{B}-\frac{1}{4}c_{B}(t-T_{n-1}(s))\leq 2C_{0}+1+\frac{3}{4}c_{B}t_{B}.
\end{split}
\end{equation*}

Now, define $\tilde{Z}(t;s):[s,\infty)\ra\R$ by setting
\begin{equation}\label{definition-new-fun}
\tilde{Z}(t;s)=Z_{n-1}(t;T_{n-1}(s))\quad\text{for}\,\, t\in[T_{n-1}(s),T_{n}(s)),\,\,n\in\N.
\end{equation}
Since $[s,\infty)=\cup_{n\in\N}[T_{n-1}(s),T_{n}(s))$, $\tilde{Z}(t;s)$ is well-defined for all $t\geq s$ (see Figure \ref{graph-modified-interface} for the illustration). Notice $\tilde{Z}(t;s)$ is strictly increasing and is linear on $[T_{n-1}(s),T_{n}(s))$ with slope $\frac{c_{B}}{2}$ for each $n\in\N$, and satisfies
\begin{equation*}
0\leq\tilde{Z}(t;s)-X_{\la_{*}}(t;s)\leq2C_{0}+1+\frac{3}{4}c_{B}t_{B},\quad t\geq s.
\end{equation*}

\begin{figure}
 \begin{center}
  \includegraphics{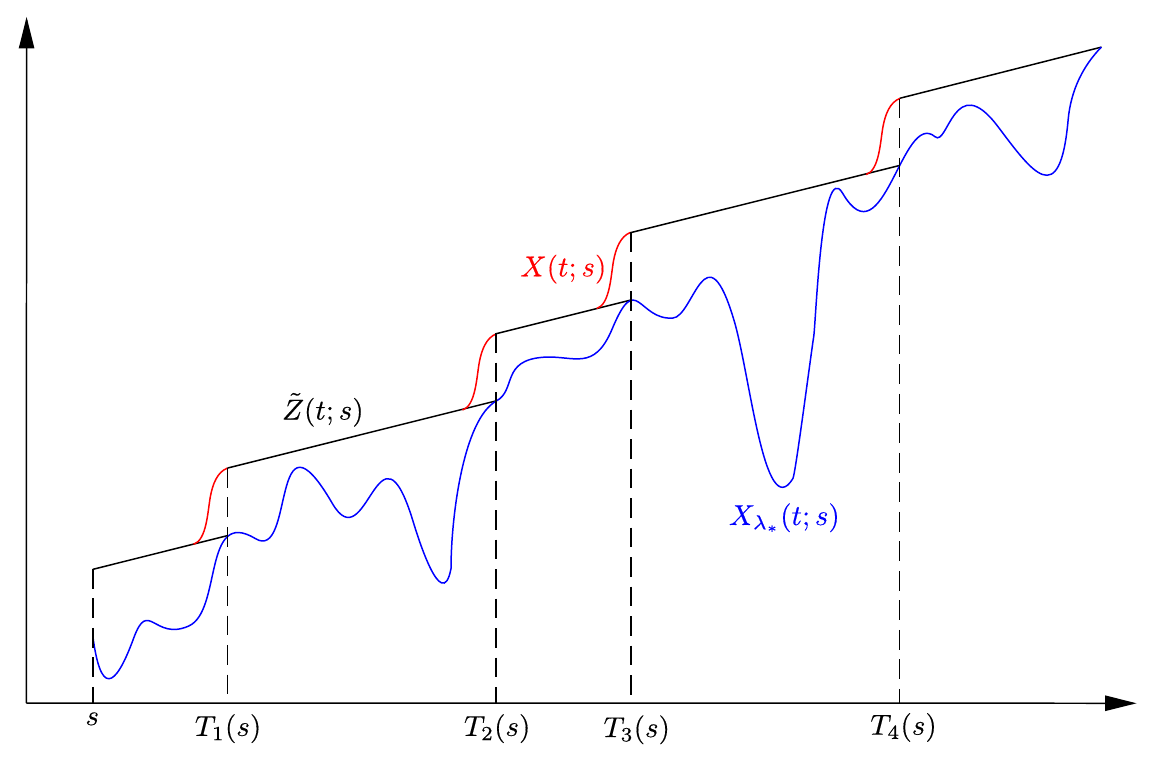}
\end{center}
\caption{\small Modified Interface Location}
  \label{graph-modified-interface}
\end{figure}

Finally, we can modify $\tilde{Z}(t;s)$ near each $T_{n}(s)$ for $n\in\N$ to get the expected modification. In fact, fix some $\de_{*}\in(0,\frac{1}{2}\frac{2}{2c_{0}-c_{B}})$. We modify $\tilde{Z}(t;s)$ by redefining it on the intervals $(T_{n}(s)-\de_{*},T_{n}(s))$, $n\in\N$ as follows: define
\begin{equation*}
X(t;s)=\begin{cases}
\tilde{Z}(t;s),\quad t\in[s,\infty)\bs\cup_{n\in\N}(T_{n}(s)-\de_{*},T_{n}(s)),\\
X_{\la_{*}}(T_{n}(s))+\de(t-T_{n}(s)),\quad t\in(T_{n}(s)-\de_{*},T_{n}(s)),\,\,n\in\N,
\end{cases}
\end{equation*}
where $\de:[-\de_{*},0]\ra[-\frac{c_{B}}{2}\de_{*},2C_{0}+1]$ is twice continuously differentiable and satisfies
\begin{equation*}
\begin{split}
&\de(-\de_{*})=-\frac{c_{B}}{2}\de_{*},\quad\de(0)=2C_{0}+1,\\
&\dot{\de}(-\de_{*})=\frac{c_{B}}{2}=\dot{\de}(0),\quad\dot{\de}(t)\geq\frac{c_{B}}{2}\,\,\text{for}\,\,t\in(-\de_{*},0)\quad\text{and}\\
&\ddot{\de}(-\de_{*})=0=\ddot{\de}(0).
\end{split}
\end{equation*}
The existence of such a function $\de(t)$ is clear. Moreover, there exist $c_{\max}=c_{\max}(\de_{*})>0$ and $\tilde{c}_{\max}=\tilde{c}_{\max}(\de_{*})>0$ such that $\dot{\de}(t)\leq c_{\max}$ and $|\ddot{\de}(t)|\leq\tilde{c}_{\max}$ for $t\in(-\de_{*},0)$. Notice the above modification is independent of $n\in\N$ and $s<0$. As a result, we readily check that $X(t;s)$ satisfies all expected properties. This completes the proof.
\end{proof}


\subsection{Uniform exponential decaying estimates}\label{subsec-decaying-estimate}

In this subsection, we study uniform exponential decaying estimates of $u(t,x;s)$ behind and ahead of interfaces. Let $\la_{*}$ be as in \eqref{a-critical-value} and $X(t;s)$ be as in Theorem \ref{thm-interface-rightward-propagation-improved}. We see that there exist $\theta_{*}\in(\theta,\la_{*}]$ and $\beta>0$ such that
\begin{equation}\label{stability-at-1}
f_{\min}(u)\geq\beta(1-u),\quad u\in[\theta_{*},1].
\end{equation}
We remark that $\beta$ may be very small.
Set
\begin{equation}\label{aux-interface-location-1}
\hat{X}(t;s)=X(t;s)-d_{\max}-\hat{C},
\end{equation}
where $\hat{C}>0$ is some constant (to be chosen) introduced only for certain flexibility. Theorem \ref{thm-bounded-interface-width} and Theorem \ref{thm-interface-rightward-propagation-improved} then imply that $\hat{X}(t;s)\leq X_{\theta_{*}}(t;s)$, and hence, $u(t,x+\hat{X}(t;s);s)\geq\theta_{*}$ for all $x\leq0$. We also set
\begin{equation}\label{aux-interface-location-2}
\tilde{X}(t;s)=X(t;s)+\sup_{s<0,t\geq s}\big|X_{\theta}(t;s)-X_{\la_{*}}(t;s)\big|.
\end{equation}
Since $X(t;s)\geq X_{\la_{*}}(t;s)$ by Theorem \ref{thm-interface-rightward-propagation-improved}, we have $\tilde{X}(t;s)\geq X_{\theta}(t;s)$, and hence, $u(t,x+\tilde{X}(t;s);s)\leq\theta$ for $x\geq0$.

We now prove the main result in this subsection.

\begin{thm}\label{thm-decaying-estimate}
There exist $c_{\pm}>0$ such that
\begin{equation*}
\begin{split}
u(t,x;s)&\geq1-e^{c_{-}(x-\hat{X}(t;s))},\quad x\leq\hat{X}(t;s),\\
u(t,x;s)&\leq e^{-c_{+}(x-\tilde{X}(t;s))},\quad x\geq\tilde{X}(t;s)
\end{split}
\end{equation*}
for all $s<0$, $t\geq s$.
\end{thm}
\begin{proof}
Define
\begin{equation*}
N_{-}[v]=v_{t}-[J\ast v-v]-\beta(1-v).
\end{equation*}
For $x\leq\hat{X}(t;s)$, we have $u(t,x;s)\geq\theta_{*}$, which together with \eqref{stability-at-1} implies that
\begin{equation*}
f(t,u(t,x;s))\geq f_{\min}(u(t,x;s))\geq\beta(1-u(t,x;s)),\quad x\leq\hat{X}(t;s)
\end{equation*}
It then follows that for $x\leq\hat{X}(t;s)$
\begin{equation*}
N_{-}[u]=u_{t}-[J\ast u-u]-\beta(1-u)=f(t,u)-\beta(1-u)\geq0.
\end{equation*}
For $c>0$, we compute
\begin{equation*}
\begin{split}
N_{-}[1-e^{c(x-\hat{X}(t;s))}]&=\bigg[c\dot{\hat{X}}(t;s)+\int_{\R}J(y)e^{-cy}dy-1-\beta\bigg]e^{c(x-\hat{X}(t;s))}\\
&\leq\bigg[cc_{\max}+\int_{\R}J(y)e^{-cy}dy-1-\beta\bigg]e^{c(x-\hat{X}(t;s))},
\end{split}
\end{equation*}
where we used the definition of $\hat{X}(t;s)$ and Theorem \ref{thm-interface-rightward-propagation-improved}. Since $\int_{\R}J(y)e^{-cy}dy\to1$ as $c\to0$, we can choose $c>0$ so small that $cc_{\max}+\int_{\R}J(y)e^{-cy}dy-1-\beta\leq0$, and then, $N_{-}[1-e^{c(x-\hat{X}(t;s))}]\leq0$. Hence, we have shown
\begin{equation*}
N_{-}[u]\geq0\geq N_{-}[1-e^{c_{-}(x-\hat{X}(t;s))}],\quad x\leq\hat{X}(t;s).
\end{equation*}
for some small $c_{-}>0$. Trivially, we have $u(t,x;s)>0\geq1-e^{c_{-}(x-\hat{X}(t;s))}$ for $x\geq\hat{X}(t;s)$. At the initial moment $s$, we obtain from Lemma \ref{lem-app-exponential-decay} that $u(s,x;s)=\phi_{\min}(x-y_{s})\geq1-e^{c_{-}(x-\hat{X}(s;s))}$ if we choose $c_{-}$ smaller and $\hat{C}$ sufficiently large. Then, we conclude from the comparison principle (see Proposition \ref{lem-comparison}$\rm(ii)$) that $u(t,x;s)\geq 1-e^{c_{-}(x-\hat{X}(t;s))}$ for $x\leq\hat{X}(t;s)$. This proves half of the theorem.

We now prove the other half. To do so, we define
\begin{equation*}
N_{+}[v]=v_{t}-[J\ast v-v].
\end{equation*}
Since $\tilde{X}(t;s)\geq X_{\theta}(t;s)$ by construction, we have $u(t,x;s)\leq\theta$ for $x\geq\tilde{X}(t;s)$, and hence, $f(t,u(t,x;s))=0$ for $x\geq\tilde{X}(t;s)$. From which, we deduce
\begin{equation*}
N_{+}[u]=u_{t}-[J\ast u-u]=f(t,u)=0,\quad x\geq\tilde{X}(t;s).
\end{equation*}
Let $c>0$. We compute
\begin{equation*}
\begin{split}
N_{+}[e^{-c(x-\tilde{X}(t;s))}]&=\bigg[c\dot{\tilde{X}}(t;s)-\int_{\R}J(y)e^{cy}dy+1\bigg]e^{-c(x-\tilde{X}(t;s))}\\
&\geq\bigg[\frac{cc_{B}}{2}-\int_{\R}J(y)e^{cy}dy+1\bigg]e^{-c(x-\tilde{X}(t;s))},
\end{split}
\end{equation*}
where we used Theorem \ref{thm-interface-rightward-propagation-improved}. Set $g(c)=\frac{cc_{B}}{2}-\int_{\R}J(y)e^{cy}dy+1$. Clearly, $g(0)=0$ and $g'(c)=\frac{c_{B}}{2}-\int_{\R}J(y)ye^{cy}dy$. Due to the symmetry of $J$, $\int_{\R}J(y)ye^{cy}dy\to0$ as $c\to0$. As a result, $g'(c)>0$ for all small $c>0$. Hence, we can find some $c_{+}>0$ such that $g(c^{+})>0$, and therefore, $N_{+}[e^{-c_{+}(x-\tilde{X}(t;s))}]\geq g(c_{+})e^{-c_{+}(x-\tilde{X}(t;s))}\geq0$. Hence, we have shown
\begin{equation*}
N_{+}[u]\leq0\leq N_{+}[e^{-c_{+}(x-\tilde{X}(t;s))}],\quad x\geq\tilde{X}(t;s).
\end{equation*}
for some small $c_{+}>0$. Since we have trivially $u(t,x;s)<1\leq e^{-c_{+}(x-\tilde{X}(t;s))}$ for $x\leq\tilde{X}(t;s)$ and, by Lemma \ref{lem-app-exponential-decay} and $\tilde{X}(s;s)\geq X_{\theta}(s;s)=y_{s}$, $u(s,x;s)=\phi_{\min}(x-y_{s})\leq e^{-c_{+}(x-\tilde{X}(t;s))}$ if we choose $c_{+}$ smaller, we conclude from the comparison principle (see Proposition \ref{lem-comparison}$\rm(i)$) that $u(t,x;s)\leq e^{-c_{+}(x-\tilde{X}(t;s))}$ for $x\geq\tilde{X}(t;s)$. This completes the proof.
\end{proof}

We remark that the fact that $\dot{X}(t;s)\geq\frac{c_{B}}{2}$ plays an essential role in deriving the uniform exponential decaying estimate of $u(t,x;s)$ at $\infty$.

As a simple consequence of Theorem \ref{thm-decaying-estimate}, we have

\begin{cor}\label{cor-bound-interface-width}
For any $0<\la_{1}<\la_{2}<1$, there holds
\begin{equation*}
\sup_{s<0,t\geq s}\big[X_{\la_{1}}(t;s)-X_{\la_{2}}(t;s)\big]<\infty.
\end{equation*}
In particular, for any $\la\in(0,1)$, there holds
\begin{equation*}
\sup_{s<0,t\geq s}\big|X_{\la}(t;s)-X(t;s)\big|<\infty.
\end{equation*}
\end{cor}
\begin{proof}
By Theorem \ref{thm-decaying-estimate}, we have
\begin{equation*}
\max\big\{0,1-e^{c_{-}(x-\hat{X}(t;s))}\big\}\leq u(t,x;s)\leq\min\big\{1,e^{-c_{+}(x-\tilde{X}(t;s))}\big\}.
\end{equation*}
The result then follows from the fact that $\tilde{X}(t;s)-\hat{X}(t;s)\equiv\text{const.}$
\end{proof}


\section{Construction of transition fronts}\label{sec-transition-front}

In this section, we prove Theorem \ref{thm-transition-front} and Theorem \ref{thm-bounded-oscillation}. To do so, we prove uniform Lipschitz continuity of $u(t,x;s)$ in the space variable $x$.

\begin{lem}\label{lem-reguarity}
Suppose $\rm(H1)$ and $\rm(H2)$. There holds
\begin{equation*}
\sup_{x\neq y\atop s<0,t\geq s}\bigg|\frac{u(t,y;s)-u(t,x;s)}{y-x}\bigg|<\infty.
\end{equation*}
\end{lem}
\begin{proof}
Since $u(t,x;s)\in(0,1)$, there holds trivially
\begin{equation*}
\forall\de>0,\quad\sup_{|y-x|\geq\de\atop s<0,t\geq s}\bigg|\frac{u(t,y;s)-u(t,x;s)}{y-x}\bigg|<\infty.
\end{equation*}
Thus, to finish the proof of the lemma, it suffices to show the local uniform Lipschitz continuity, that is,
\begin{equation}\label{locally-uniform-Lip-continuity}
\forall\de>0,\quad\sup_{0<|y-x|\leq\de\atop s<0,t\geq s}\bigg|\frac{u(t,y;s)-u(t,x;s)}{y-x}\bigg|<\infty.
\end{equation}

To this end, we fix $\de>0$. Let $X(t;s)$ be as in Theorem \ref{thm-interface-rightward-propagation-improved} and define
\begin{equation*}
L_{1}=\de+\sup_{s<0,t\geq s}\big|X_{\theta}(t;s)-X(t;s)\big|\quad\text{and}\quad
L_{2}=\de+\sup_{s<0,t\geq s}\big|X_{\tilde{\theta}}(t;s)-X(t;s)\big|,
\end{equation*}
where $\tilde{\theta}\in(\theta,1)$ is given in $\rm(H2)$. Notice $L_{1}<\infty$ and $L_{2}<\infty$ by Corollary \ref{cor-bound-interface-width}. Then, for any $0<|y-x|\leq\de$ we have
\begin{itemize}
\item if $x\geq X(t;s)+L_{1}$, then $y\geq x-\de\geq X_{\theta}(t;s)$, which implies that $f(t,u(t,y;s))=0=f(t,u(t,x;s))$,
and hence
\begin{equation}\label{estimate-aux-1}
\frac{f(t,u(t,y;s))-f(t,u(t,x;s))}{u(t,y;s)-u(t,x;s)}=0;
\end{equation}

\item if $x\leq X(t;s)-L_{2}$, then $y\leq x+\de\leq X_{\tilde{\theta}}(t;s)$, which implies that $u(t,y;s)\geq\tilde{\theta}$ and $u(t,x;s)\geq\tilde{\theta}$, and hence by $\rm(H2)$
\begin{equation}\label{estimate-aux-2}
\frac{f(t,u(t,y;s))-f(t,u(t,x;s))}{u(t,y;s)-u(t,x;s)}\leq0.
\end{equation}
\end{itemize}
According to \eqref{estimate-aux-1} and \eqref{estimate-aux-2}, we consider time-dependent disjoint decompositions of $\R$ into
\begin{equation*}
\R=R_{l}(t;s)\cup R_{m}(t;s)\cup R_{r}(t;s),
\end{equation*}
where
\begin{equation*}
\begin{split}
R_{l}(t;s)&=(-\infty,X(t;s)-L_{2}),\\
R_{m}(t;s)&=[X(t;s)-L_{2}, X(t;s)+L_{1}]\quad\text{and}\\
R_{r}(t;s)&=(X(t;s)+L_{1},\infty).
\end{split}
\end{equation*}
Since $X(t;s)$ is continuous in $t$, this three regions change continuously in $t$. As $X(t;s)$ moves to the right by Theorem \ref{thm-interface-rightward-propagation-improved}, any fixed point will eventually enter into $R_{l}(t;s)$ and stay there forever.

For $s<0$ and $x_{0}\in\R$, let $t_{\rm first}(x_{0};s)$ be the first time that $x_{0}$ is in $R_{m}(t;s)$, that is,
\begin{equation*}
t_{\rm first}(x_{0};s)=\min\big\{t\geq s\big|x_{0}\in R_{m}(t;s)\big\},
\end{equation*}
and $t_{\rm last}(x_{0};s)$ be the last time that $x_{0}$ is in $R_{m}(t;s)$, that is,
\begin{equation*}
t_{\rm last}(x_{0};s)=\max\big\{t_{0}\in\R\big|x_{0}\in R_{m}(t_{0};s)\,\,\text{and}\,\,x_{0}\notin R_{m}(t,s)\,\,\text{for}\,\,t>t_{0}\big\}.
\end{equation*}
Since $X(t;s)$ moves to the right, if $x_{0}\in R_{l}(s;s)$, then $x_{0}\in R_{l}(t;s)$ for all $t>s$. In this case, $t_{\rm first}(x_{0};s)$ and $t_{\rm last}(x_{0};s)$ are not well-defined, but it will not cause any trouble. Clearly, $x_{0}\in R_{l}(t;s)$ for all $t>t_{\rm last}(x_{0};s)$.

We see that either both $t_{\rm first}(x_{0};s)$ and $t_{\rm last}(x_{0};s)$ are well-defined, or both of them are not well-defined. In fact, $t_{\rm first}(x_{0};s)$ and $t_{\rm last}(x_{0};s)$ are well-defined only if $x_{0}\notin R_{l}(s;s)$. As a simple consequence of Theorem \ref{thm-interface-rightward-propagation-improved} and the fact that the length of $R_{m}(t;s)$ is $L_{1}+L_{2}$, we have
\begin{equation}\label{growth-period}
T=T(\de):=\sup_{s<0,x_{0}\notin R_{l}(s;s)}\big[t_{\rm last}(x_{0};s)-t_{\rm first}(x_{0};s)\big]<\infty.
\end{equation}

Now, we are ready to prove the lemma. Fix $x_{0}\in\R$, $s<0$ and $0<|\eta|\leq\de$. Set
\begin{equation*}
v^{\eta}(t,x;s)=\frac{u(t,x+\eta;s)-u(t,x;s)}{\eta}.
\end{equation*}
Clearly, $v^{\eta}(t,x_{0};s)$ satisfies
\begin{equation*}
v^{\eta}_{t}(t,x_{0};s)=\int_{\R}J(x_{0}-y)v^{\eta}(t,y;s)dy-v^{\eta}(t,x_{0};s)+a^{\eta}(t,x_{0};s)v^{\eta}(t,x_{0};s),
\end{equation*}
where
\begin{equation*}
a^{\eta}(t,x_{0};s)=\frac{f(t,u(t,x_{0}+\eta;s))-f(t,u(t,x_{0};s))}{u(t,x_{0}+\eta;s)-u(t,x_{0};s)}
\end{equation*}
is uniformly bounded. Notice $\int_{\R}J(x_{0}-y)v^{\eta}(t,y;s)dy$ is bounded uniformly in $x_{0}$, $\eta$, $s$ and $t$. In fact, the change of variable gives
\begin{equation*}
\int_{\R}J(x_{0}-y)v^{\eta}(t,y;s)dy=\int_{\R}\frac{J(x_{0}-y+\eta)-J(x_{0}-y)}{\eta}u(t,y;s)dy.
\end{equation*}
The uniform boundedness then follows from the fact $u(t,x;s)\in(0,1)$ and the assumption $J'\in L^{1}(\R)$ by $\rm(H1)$.

Setting 
\begin{equation*}
M=M(\de):=\sup_{s<0,t\geq s}\sup_{x_{0}\in\R,0<|\eta|\leq\de}\bigg|\int_{\R}J(x_{0}-y)v^{\eta}(t,y;s)dy\bigg|, 
\end{equation*}
we see that $v^{\eta}(t,x_{0};s)$ satisfies
\begin{equation}\label{ODIs}
\begin{split}
&-M-v^{\eta}(t,x_{0};s)+a^{\eta}(t,x_{0};s)v^{\eta}(t,x_{0};s)\\
&\quad\quad\leq v^{\eta}_{t}(t,x_{0};s)\leq M-v^{\eta}(t,x_{0};s)+a^{\eta}(t,x_{0};s)v^{\eta}(t,x_{0};s),
\end{split}
\end{equation}
which essentially are ordinary differential inequalities. The solution of \eqref{ODIs} satisfies
\begin{equation}\label{ODIs-sol}
\begin{split}
&v^{\eta}(t_{0},x_{0};s)e^{-\int_{t_{0}}^{t}(1-a^{\eta}(\tau,x_{0};s))d\tau}-M\int_{t_{0}}^{t}e^{-\int_{r}^{t}(1-a^{\eta}(\tau,x_{0};s))d\tau}dr\\
&\quad\quad\leq v^{\eta}(t,x_{0};s)\leq v^{\eta}(t_{0},x_{0};s)e^{-\int_{t_{0}}^{t}(1-a^{\eta}(\tau,x_{0};s))d\tau}+M\int_{t_{0}}^{t}e^{-\int_{r}^{t}(1-a^{\eta}(\tau,x_{0};s))d\tau}dr
\end{split}
\end{equation}
for $s\leq t_{0}\leq t$. Notice $1-a^{\eta}(t,x_{0};s)$ controls the behavior of $v^{\eta}(t,x_{0};s)$. We see, initially,
\begin{equation*}
v^{\eta}(s,x_{0};s)=\frac{u(s,x_{0}+\eta;s)-u(s,x_{0};s)}{\eta}=\frac{\phi_{\min}(x_{0}+\eta-y_{s})-\phi_{\min}(x_{0}-y_{s})}{\eta},
\end{equation*}
which is uniformly bounded (uniform in $s<0$, $x_{0}\in\R$ and $0<|\eta|\leq\de$, and even in $\de>0$) since $\phi_{\min}\in C^{1}(\R)$. Now,
\begin{itemize}
\item[\rm(i)] if $x_{0}\in R_{l}(s;s)$, then $x_{0}\in R_{l}(t;s)$ for all $t\geq s$, which implies that $a^{\eta}(t,x_{0};s)\leq0$ for all $t\geq s$ by \eqref{estimate-aux-2}. We then conclude from \eqref{ODIs-sol} that $\sup_{t\geq s}|v^{\eta}(t,x_{0};s)|\leq C(M,\|\phi_{\min}\|_{C^{1}(\R)})$;

\item[\rm(ii)] if $x_{0}\in R_{m}(s;s)$, then $t_{\rm first}(x_{0};s)=s$. For $t\in[t_{\rm first}(x_{0};s),t_{\rm last}(x_{0};s)]$, we conclude from \eqref{ODIs-sol} and the fact that $a^{\eta}(t,x_{0};s)$ is uniformly bounded that $v^{\eta}(t,x_{0};s)$ at most grows exponentially with growth rate not larger than some universal constant, and this expnential growth can only last for a period not longer than $T$, which is given in \eqref{growth-period}. As a result, $|v^{\eta}(t_{\rm last}(x_{0};s),x_{0};s)|\leq C(M,\|\phi_{\min}\|_{C^{1}(\R)},T)$. As mentioned before, for $t>t_{\rm last}(x_{0};s)$, we have $x_{0}\in R_{l}(t;s)$, which together with \eqref{estimate-aux-2} implies that $a^{\eta}(t,x_{0};s)\leq0$. Then, as in $\rm(i)$, we conclude from  \eqref{ODIs-sol} that $|v^{\eta}(t,x_{0};s)|\leq C(M,\|\phi_{\min}\|_{C^{1}(\R)},T)$ for $t>t_{\rm last}(x_{0};s)$. Hence, $\sup_{t\geq s}|v^{\eta}(t,x_{0};s)|\leq C(M,\|\phi_{\min}\|_{C^{1}(\R)},T)$;

\item[\rm(iii)] if $x_{0}\in R_{r}(s;s)$, then $t_{\rm first}(x_{0};s)>s$. For $t\in[s,t_{\rm first}(x_{0};s))$, we have $x_{0}\in R_{r}(t;s)$, which implies $a^{\eta}(t,x_{0};s)=0$ by \eqref{estimate-aux-1}. Notice the interval $[s,t_{\rm first}(x_{0};s))$ may not have uniformly bounded length, but \eqref{ODIs-sol} says that as long as $t\in[s,t_{\rm first}(x_{0};s))$, we have $|v^{\eta}(t,x_{0};s)|\leq C(M,\|\phi_{\min}\|_{C^{1}(\R)})$, which implies, $v^{\eta}(t_{\rm first}(x_{0};s),x_{0};s)\leq C(M,\|\phi_{\min}\|_{C^{1}(\R)})$. Then, we can follow the arguments in $\rm(ii)$ to conclude that $\sup_{t\geq s}|v^{\eta}(t,x_{0};s)|\leq C(M,\|\phi_{\min}\|_{C^{1}(\R)},T)$.
\end{itemize}
Consequently, $\sup_{t\geq s}|v^{\eta}(t,x_{0};s)|\leq C(M,\|\phi_{\min}\|_{C^{1}(\R)},T)$. Since $x_{0}\in\R$ and $s<0$ are arbitrary, and $T$ and $M$ depend only on $\de$, we find \eqref{locally-uniform-Lip-continuity}, and hence, finish the proof of the lemma.
\end{proof}

Now, we prove Theorem \ref{thm-transition-front}.

\begin{proof}[Proof of Theorem \ref{thm-transition-front}]
Since $u=u(t,x;s)$ satisfies $u_{t}=J\ast u-u+f(t,u)$, we conclude from $\rm(H2)$ and the fact that $u(t,x;s)\in(0,1)$ that
\begin{equation}\label{1-bound-time-derivative}
\sup_{s<0,t>s}|u_{t}(t,x;s)|<\infty.
\end{equation}
Then, since $v(t,x;s):=u_{t}(t,x;s)$ satisfies
\begin{equation*}
v_{t}=J\ast v-v+f_{t}(t,u(t,x;s))+f_{u}(t,u(t,x;s))v,
\end{equation*}
we conclude from \eqref{1-bound-time-derivative} and $\rm(H2)$ that
\begin{equation}\label{2-bound-time-derivative}
\sup_{s<0,t>s}|u_{tt}(t,x;s)|=\sup_{s<0,t>s}|v_{t}(t,x;s)|<\infty.
\end{equation}
Now, by Lemma \ref{lem-reguarity}, \eqref{1-bound-time-derivative}, \eqref{2-bound-time-derivative}, Arzel\`{a}-Ascoli theorem and the diagonal argument, we are able to find some continuous function $u(t,x):\R\times\R\ra[0,1]$ that is differentiable in $t$ and nonincreasing in $x$ such that
\begin{equation*}
u(t,x;s)\to u(t,x)\quad\text{and} \quad u_{t}(t,x;s)\to u_{t}(t,x)
\end{equation*}
locally uniformly in $(t,x)\in\R\times\R$ as $s\ra-\infty$ along some subsequence. In particular, $u(t,x)$ is an entire classical solution of \eqref{main-eqn}. Moreover, as an entire solution, $u(t,x)\in(0,1)$ and it is strictly decreasing in $x$. The uniform Lipschitz continuity in space of $u(t,x)$ also follows.

We now show the decaying properties of $u(t,x)$. Recall $\hat{X}(t;s)$ and $\tilde{X}(t;s)$ are given in \eqref{aux-interface-location-1} and \eqref{aux-interface-location-2}, respectively. By Theorem \ref{thm-interface-rightward-propagation-improved}, $\{X(t;s)\}_{s<0}$, $\{\hat{X}(t;s)\}_{s<0}$ and $\{\tilde{X}(t;s)\}_{s<0}$ converge locally uniformly to a continuously differentiable functions $X(t)$, $\hat{X}(t)$ and $\tilde{X}(t)$, respectively. Clearly, $\frac{c_{B}}{2}\leq\dot{X}(t)\leq c_{\max}$, $X(t)-\hat{X}(t)=h_{-}$ and $\tilde{X}-X(t)=h_{+}$ for all $t\in\R$, where $h_{\pm}>0$ are constants. In particular,
\begin{equation*}
u(t,x+\hat{X}(t;s);s)\to u(t,x+\hat{X}(t))\quad\text{and}\quad u(t,x+\tilde{X}(t;s);s)\to u(t,x+\tilde{X}(t))
\end{equation*}
locally uniformly in $(t,x)\in\R\times\R$ as $s\ra-\infty$ along some subsequence, which together with Theorem \ref{thm-decaying-estimate} implies that
\begin{equation}\label{transition-front-exponential-decay}
\begin{split}
u(t,x+\hat{X}(t))&\geq1-e^{c_{-}x},\quad x\leq0,\\
u(t,x+\tilde{X}(t))&\leq e^{-c_{+}x},\quad x\geq0
\end{split}
\end{equation}
for all $t\in\R$. This completes the proof.
\end{proof}

Finally, we prove Theorem \ref{thm-bounded-oscillation}.

\begin{proof}[Proof of Theorem \ref{thm-bounded-oscillation}]
We rewrite \eqref{transition-front-exponential-decay} as
\begin{equation}\label{transition-front-exponential-decay-1}
1-e^{c_{-}(x-\hat{X}(t))}\leq u(t,x)\leq e^{-c_{+}(x-\tilde{X}(t))},\quad x\in\R,\,\,t\in\R.
\end{equation}
For $\la\in(0,1)$, let $\hat{X}_{\la}(t)$ be the unique solution of $1-e^{c_{-}(x-\hat{X}(t))}=\la$, and $\tilde{X}_{\la}(t)$ be the unique solution of $e^{-c_{+}(x-\tilde{X}(t))}=\la$. Clearly,
\begin{equation}\label{interface-location-at-la-aux}
\begin{split}
\hat{X}_{\la}(t)&=\hat{X}(t)+\frac{\ln(1-\la)}{c_{-}}=X(t)-h_{-}+\frac{\ln(1-\la)}{c_{-}},\\
\tilde{X}_{\la}(t)&=\tilde{X}(t)-\frac{\ln\la}{c_{+}}=X(t)+h_{+}-\frac{\ln\la}{c_{+}}.
\end{split}
\end{equation}
Let $X_{\la}(t)$ be such that $u(t,X_{\la}(t))=\la$. From \eqref{transition-front-exponential-decay-1}, we see that $X_{\la}(t)\in[\hat{X}_{\la}(t),\tilde{X}_{\la}(t)]$, which together with \eqref{interface-location-at-la-aux} implies that
\begin{equation}\label{bounded-interface-width-transition-front}
\sup_{t\in\R}|X_{\la}(t)-X(t)|\leq h(\la):=\max\bigg\{h_{-}-\frac{\ln(1-\la)}{c_{-}},h_{+}-\frac{\ln\la}{c_{+}}\bigg\}.
\end{equation}

Set $\bar{C}:=\sup_{(t,x)\in\R\times\R}|u_{t}(t,x)|<\infty$. For any $t\in\R$ and $\eta>0$, we have
\begin{equation*}
u(t+\eta,X_{\theta}(t))-\theta=u(t+\eta,X_{\theta}(t))-u(t,X_{\theta}(t))=u_{t}(t+\bar{\eta},X_{\theta}(t))\eta
\end{equation*}
for some $\bar{\eta}\in[0,\eta]$, which implies that $u(t+\eta,X_{\theta}(t))\in[\theta-\bar{C}\eta,\theta+\bar{C}\eta]$. Let $\eta_{0}>0$ be such that $\bar{C}\eta_{0}=\min\big\{\frac{\theta}{2},\frac{1-\theta}{2}\big\}$, we see that
\begin{equation*}
u(t+\eta,X_{\theta}(t))\in\bigg[\frac{\theta}{2},\frac{1+\theta}{2}\bigg],\quad t\in\R,\,\,\eta\in[0,\eta_{0}],
\end{equation*}
which together with space monotonicity implies that
\begin{equation*}
X_{\frac{1+\theta}{2}}(t+\eta)\leq X_{\theta}(t)\leq X_{\frac{\theta}{2}}(t+\eta),\quad t\in\R,\,\,\eta\in[0,\eta_{0}].
\end{equation*}
Clearly, as a simple consequence of \eqref{bounded-interface-width-transition-front}, we have
\begin{equation*}
\begin{split}
\sup_{t\in\R,\eta\in[0,\eta_{0}]}|X_{\theta}(t+\eta)-X_{\frac{\theta}{2}}(t+\eta)|&\leq h(\theta)+h(\frac{\theta}{2}),\\
\sup_{t\in\R,\eta\in[0,\eta_{0}]}|X_{\theta}(t+\eta)-X_{\frac{1+\theta}{2}}(t+\eta)|&\leq h(\theta)+h(\frac{1+\theta}{2}).
\end{split}
\end{equation*}
Hence, we deduce
\begin{equation*}
\sup_{t\in\R,\eta\in[0,\eta_{0}]}|X_{\theta}(t+\eta)-X_{\theta}(t)|\leq C_{1}:=\max\bigg\{h(\theta)+h(\frac{\theta}{2}),h(\theta)+h(\frac{1+\theta}{2})\bigg\}.
\end{equation*}
It then follows that for any $\de=k\eta_{0}+\de_{0}$ with $k\in\N_{0}$ and $\de_{0}\in[0,\eta_{0})$, we have
\begin{equation}\label{bounded-oscillation-estimate}
\sup_{|t-s|\leq\de}|X_{\theta}(t)-X_{\theta}(s)|\leq(k+1)C_{1}.
\end{equation}

The "in particular" part follows if we set $\de=1$ in \eqref{bounded-oscillation-estimate} and iterate. It also immediately follows from that the facts that $\sup_{t\in\R}|X_{\theta}(t)-X(t)|<\infty$ and $\dot{X}(t)\in[\frac{c_{B}}{2},c_{\max}]$ for all $t\in\R$. This completes the proof.
\end{proof}


\section*{Acknowledgements}

The authors would like to thank the referees for carefully reading the original manuscript and providing helpful suggestions.


\appendix

\section{Comparison principles}\label{sec-app-CP}

We prove comparison principles used in the previous sections.

\begin{prop}\label{lem-comparison}
Let $K:\R\times\R\to[0,\infty)$ be continuous and satisfy $\sup_{x\in\R}\int_{\R}K(x,y)dy<\infty$.
Let  $a:\R\times\R\to\R$ be continuous and uniformly bounded.

\begin{itemize}
\item[\rm(i)] Suppose that $X:[0,\infty)\to\R$ is continuous and that $u:[0,\infty)\times\R\to\R$ satisfies the following:   $u, u_t:[0,\infty)\times\R\to\R$
are continuous, the limit $\lim_{x\ra\infty}u(t,x)=0$ is locally uniformly in $t$,  and
\begin{equation*}
\begin{cases}
u_{t}(t,x)\geq \int_{\R}K(x,y)u(t,y)dy+a(t,x)u(t,x),\quad x>X(t),\,\,t>0,\\
u(t,x)\geq0,\quad x\leq X(t),\,\,t>0,\\
u(0,x)=u_{0}(x)\geq0,\quad x\in\R.
\end{cases}
\end{equation*}
Then $u(t,x)\geq0$ for $(t,x)\in(0,\infty)\times\R$.

\item[\rm(ii)]  Suppose that $X:[0,\infty)\to\R$ is continuous and that $u:[0,\infty)\times\R\to\R$ satisfies the following:   $u, u_t:[0,\infty)\times\R\to\R$
are continuous, the limit $\lim_{x\ra -\infty}u(t,x)=0$ is locally uniformly in $t$,  and
\begin{equation*}
\begin{cases}
u_{t}(t,x)\geq \int_{\R}K(x,y)u(t,y)dy+a(t,x)u(t,x),\quad x<X(t),\,\,t>0,\\
u(t,x)\geq0,\quad x\ge X(t),\,\,t>0,\\
u(0,x)=u_{0}(x)\geq0,\quad x\in\R.
\end{cases}
\end{equation*}
Then $u(t,x)\geq0$ for $(t,x)\in(0,\infty)\times\R$.

\item[\rm(iii)] Suppose  that $u:[0,\infty)\times\R\to\R$ satisfies the following:   $u, u_t:[0,\infty)\times\R\to\R$
is continuous, $\inf_{t\ge 0,x\in\R}u(t,x)>-\infty$, and
\begin{equation*}
\begin{cases}
u_{t}(t,x)\geq \int_{\R}K(x,y)u(t,y)dy+a(t,x)u(t,x),\quad x\in\R,\,\,t>0,\\
u(0,x)=u_{0}(x)\geq0,\quad x\in\R.
\end{cases}
\end{equation*}
Then $u(t,x)\geq0$ for $(t,x)\in(0,\infty)\times\R$. Moreover, if $u_0(x)\not\equiv 0$, then
$u(t,x)>0$ for $(t,x)\in(0,\infty)\times\R$.
\end{itemize}
\end{prop}
\begin{proof}
$\rm(i)$ We follow \cite[Proposition 2.4]{HSV08}. Note first that replacing $u(t,x)$ by $e^{rt}u(t,x)$ for sufficiently large $r>0$, we may assume, without loss of generality, that $a(t,x)>0$ for all $(t,x)\in\R\times\R$.

Set $K_{0}:=\sup_{x\in\R}\int_{\R}K(x,y)dy<\infty$, $a_{0}:=\sup_{(t,x)\in\R\times\R}a(t,x)$ and $\tau:=\frac{1}{2(K_{0}+a_{0})}$. Suppose there exists $(t_{0},x_{0})\in[0,\tau]\times\R$ such that $u(t_{0},x_{0})<0$. Then, by the assumption, there exists $(t_{1},x_{1})\in\Om_{0,\tau}:=\{(t,x)\in\R\times\R|x>X(t),\,\,t\in(0,\tau]\}$ such that
\begin{equation*}
u(t_{1},x_{1})=\min_{(t,x)\in\Om_{0,\tau}}u(t,x)<0.
\end{equation*}

Define
\begin{equation*}
s_{1}:=\max\{t\in[0,t_{1}]|u(t,x_{1})\geq0\}.
\end{equation*}
By continuity of $u(t,x)$, $s_{1}<t_{1}$ and  $u(s_{1},x_{1})\geq0$. Moreover, by the definition of $s_{1}$, we see that $x_{1}>X(t)$ for $t\in(s_{1},t_{1}]$. In particular, we have
\begin{equation*}
u_{t}(t,x_{1})\geq \int_{\R}K(x_{1},y)u(t,y)dy+a(t,x_{1})u(t,x_{1}),\quad t\in(s_{1},t_{1}].
\end{equation*}
Integrating the above inequality with respect to $t$ from $s_{1}$ to $t_{1}$, we conclude from the facts that $u(t,x)\geq0$ for $x\leq X(t)$ and $u(t_{1},x_{1})<0$ that
\begin{equation*}
\begin{split}
u(t_{1},x_{1})-u(s_{1},x_{1})&\geq\int_{s_{1}}^{t_{1}}\int_{\R}K(x_{1},y)u(t,y)dydt+\int_{s_{1}}^{t_{1}}a(t,x_{1})u(t,x_{1})dt\\
&\geq\int_{s_{1}}^{t_{1}}\int_{X(t)}^{\infty}K(x_{1},y)u(t,y)dydt+\int_{s_{1}}^{t_{1}}a(t,x_{1})u(t,x_{1})dt\\
&\geq u(t_{1},x_{1})\bigg[\int_{s_{1}}^{t_{1}}\int_{X(t)}^{\infty}K(x_{1},y)dydt+\int_{s_{1}}^{t_{1}}a(t,x_{1})dt\bigg]\\
&\geq u(t_{1},x_{1})(K_{0}+a_{0})(t_{1}-s_{1})\\
&\geq u(t_{1},x_{1})(K_{0}+a_{0})\tau,
\end{split}
\end{equation*}
which implies that $[1-(K_{0}+a_{0})\tau]u(t_{1},x_{1})\geq u(s_{1},x_{1})\geq0$. It then follows from the choice of $\tau$ that $u(t_{1},x_{1})\geq0$. It's a contradiction. Thus, $u(t,x)\geq0$ for $(t,x)\in[0,\tau]\times\R$. Repeating the above arguments with initial times $\tau,2\tau,3\tau,\dots$, we find the result.

$\rm(ii)$ It can be proved by the similar arguments as in $\rm(i)$.

$\rm(iii)$ It follows from the arguments in \cite[Propositions 2.1 and 2.2]{ShZh10}.
\end{proof}


\bibliographystyle{amsplain}

\end{document}